\documentclass[a4paper,10pt]{article}

\usepackage{amsmath}
\usepackage{amsthm}
\usepackage{amsfonts}
\usepackage[titletoc,title]{appendix}
\usepackage[backend=bibtex8,doi=false,eprint=false,firstinits=true,isbn=false,style=numeric-comp,url=false,maxnames=99]{biblatex}
\bibliography{bibl.bib}
\DeclareRedundantLanguages{english,german,french}{english,german,ngerman,french}

\usepackage{constants}
\usepackage{bbm}

\usepackage{enumerate}

\usepackage{color}
\usepackage[colorlinks]{hyperref}
\definecolor{blue75}{rgb}{0,0,.75}
\definecolor{green75}{rgb}{0,.75,0}
\hypersetup{colorlinks=true, urlcolor=blue75,linkcolor=blue75,citecolor=green75,pdfstartview=FitB,bookmarksopen=true,bookmarksopenlevel=1}

\newcommand{\cb }{}
\usepackage[a4paper, left=2.5cm, right=2.5cm, top=2.5cm,bottom=2cm]{geometry}

\newcommand{\parenthezises}[1]{\arabic{#1}}
\newconstantfamily{A}{
symbol=A,
format=\parenthezises,
reset={section}
}
\newconstantfamily{B}{
symbol=B,
format=\parenthezises,
reset={section}
}
\newconstantfamily{b}{
symbol=b,
format=\parenthezises,
reset={section}
}
\newconstantfamily{C}{
symbol=C,
format=\parenthezises,
reset={section}
}
\newconstantfamily{D}{
symbol=D,
format=\parenthezises,
reset={section}
}
\newconstantfamily{d}{
symbol=d,
format=\parenthezises,
reset={section}
}
\newconstantfamily{F}{
symbol=F,
format=\parenthezises,
reset={section}
}
\newconstantfamily{G}{
symbol=G,
format=\parenthezises,
reset={section}
}
\newconstantfamily{U}{
symbol=U,
format=\parenthezises,
reset={section}
}
\newconstantfamily{P}{
symbol=P,
format=\parenthezises,
reset={section}
}
\newconstantfamily{z}{
symbol=\zeta,
format=\parenthezises,
reset={section}
}
\newconstantfamily{Phi}{
symbol=\varphi,
format=\parenthezises,
reset={section}
}
\newconstantfamily{theta}{
symbol=\theta,
format=\parenthezises,
reset={section}
}

\usepackage{extpfeil}
\allowdisplaybreaks
\begin{document}

\pagestyle{plain} 

\numberwithin{equation}{section}

\newcommand{\longrightharpoonup}{\rm - \!\!\! \rightharpoonup}

\newcommand{\R}{\mathbb{R}}
\newcommand{\CC}{\mathbb{C}}
\newcommand{\N}{\mathbb{N}}

\newcommand{\LOne}{L^{1}(\Omega)}
\newcommand{\LTwo}{L^{2}(\Omega)}
\newcommand{\LTwon}{(L^{2}(\Omega))^n}
\newcommand{\Lq}{L^{q}(\Omega)}
\newcommand{\Lp}{L^{p}(\Omega)}
\newcommand{\Ld}{L^\delta(\Omega)}
\newcommand{\LInf}{L^{\infty}(\Omega)}

\newcommand{\TestF}{C^{\infty}_{{0}}(\Omega)}
\newcommand{\DD}{C_0^\infty[0,T]}

\newcommand{\HOneO}{H^{1}_{0}(\Omega)}
\newcommand{\HTwoO}{H^{2}_{0}(\Omega)}

\newcommand{\HOne}{H^{1}(\Omega)}
\newcommand{\HTwo}{H^{2}(\Omega)}
\newcommand{\HFour}{H^{4}(\Omega)}
\newcommand{\Hl}{H^{l}(\Omega)}

\newcommand{\HmOne}{H^{-1}(\Omega)}
\newcommand{\HmTwo}{H^{-2}(\Omega)}

\newcommand{\TLTwoLTwo}{L^{2}\left([0,T];\LTwo\right)}

\newcommand{\THOneLTwo}{H^{1}\left([0,T];\LTwo\right)}

\newcommand{\THOneHFour}{H^{1}\left([0,T];\HFour\right)}

\newcommand{\TLTwoHOneO}{L^{2}\left([0,T];H^{1}_{0}(\Omega)\right)}
\newcommand{\TLTwoHTwoO}{L^{2}\left([0,T];\HTwoO\right)}
\newcommand{\TLTwoHmOne}{L^{2}\left([0,T];H^{-1}(\Omega)\right)}

\newcommand{\LInfLd}{L^{\infty}([0,\infty);L^{\delta}(\Omega))}
\newcommand{\TLInfLTwo}{L^{\infty}([0,T];\LTwo)}
\newcommand{\LInfQT}{L^{\infty}\left((0,T)\times(\Omega)\right)}

\newcommand{\THOneHOneO}{H^{1}\left([0,T];H^{1}_{0}(\Omega)\right)}

\newcommand{\TCLTwo}{C\left([0,T];\LTwo\right)}
\newcommand{\TCOneLTwo}{C^1\left([0,T];\LTwo\right)}

\newcommand{\LbOne}{L^1_b(\R)}
\newcommand{\LbTwo}{L^2_b(\R)}

\newcommand{\LlocTwo}{L^2_{loc}(\R)}

\newcommand{\TLlocTwoLTwo}{L^2_{loc}(\R,\LTwo)}
\newcommand{\TLbTwoHOne}{L^2_{b}(\R,\HOne)}
\newcommand{\TLlocTwoHmOne}{L^2_{loc}(\R,H^{-1}(\Omega))}
\newcommand{\TLbTwoHmOne}{L^2_{b}(\R,H^{-1}(\Omega))}

\newcommand{\TLbTwoLTwo}{L^2_{b}(\R,\LTwo)}

\newcommand{\X}{\LInf\times W^{1,\infty}(\Omega)}

\newcommand{\dist}{\operatorname{dist}}
\newcommand{\supp}{\operatorname{supp}}
\newcommand{\ess}{\operatorname{ess}}

\newtheorem{Definition}{Definition}
\newtheorem{Corollary}{Corollary}
\newtheorem{Lemma}{Lemma}
\newtheorem*{Lem}{Lemma 1}
\newtheorem{Remark}{Remark}
\newtheorem{Proposition}{Proposition}
\newtheorem{Theorem}{Theorem}
\newtheorem{Notation}{Notation}

\theoremstyle{definition}
\newtheorem{Example}{Example}



\newcommand{\Keywords}[1]{\par\noindent 
{\small{\em {\bfseries Keywords}\/}: #1}}

\title{On an exponential attractor for a class of PDEs with degenerate diffusion and chemotaxis
}
\author{Messoud Efendiev\\
\small{Helmholtz Center Munich,
  Institute of Computational Biology}\\
\small{Ingolst\"adter Landstr. 1, 85764 Neuherberg, Germany}\\
\small{e-mail: \href{mailto:messoud.efendiyev@helmholtz-muenchen.de}
{messoud.efendiyev@helmholtz-muenchen.de}}\\\\and\\\\
	Anna Zhigun\\
 \small{Technische Universit\"at Kaiserslautern, Felix-Klein-Zentrum f\"ur Mathematik}\\
\small{Paul-Ehrlich-Str. 31, 67663 Kaiserslautern, Germany}\\
\small{e-mail: \href{mailto:zhigun@mathematik.uni-kl.de}
{zhigun@mathematik.uni-kl.de}}}
\date{}
\maketitle
\begin{abstract} 
In this article we deal with a class of strongly coupled parabolic systems that encompasses two different effects: degenerate diffusion and chemotaxis. Such classes of equations arise in the mesoscale level modeling of biomass spreading
mechanisms via chemotaxis. We show the existence of an exponential attractor and, hence, of a finite-dimensional global attractor under certain 'balance conditions' on the order of the  degeneracy and the growth of the chemotactic function.\\\\
{\bf Keywords}: attractor; biofilm; chemotaxis; degenerate diffusion, longtime dynamics\\\\
2010 Mathematics Subject Classification:
2010 MSC: 35B41, 35B45, 35D30, 35K65 
\end{abstract}

\section{Introduction}
In this work we continue our studies of  the longtime behaviour of a degenerate system modelling a biomass spreading in the presence of chemotaxis which was introduced in \cite{ESe}:
\begin{subequations}\label{PMCT}
\begin{alignat}{3}
	&\partial_tM= \nabla\cdot\left(M^\alpha \nabla M-M^\gamma \nabla\rho\right)-f(M,\rho)&& \quad\text{ in }\R^+\times\Omega,\label{EqI}\\
	&\partial_t\rho=\Delta\rho-g(M,\rho)&& \quad\text{ in }\R^+\times\Omega,\label{EqII}\\
	&M=0,\quad  \rho=1 && \quad\text{ in }\R^+\times\partial\Omega, \\
	&M =M_{0},\quad \rho = \rho_{0} && \quad\text{ in }\{0\}\times\Omega,\label{rho0}
\end{alignat}
\end{subequations}
where $\Omega$ is smooth bounded domain in $\R^N$, $N\in\{1,2,3\}$, $\alpha$ and $\gamma$ are two positive constants satisfying certain conditions (we call them 'balance conditions') to be specified below, 
and $M_{0}\in L^{\infty}(\Omega)$, $\rho_{0}\in W^{1,\infty}(\Omega)$ are nonnegative functions.

Equations \eqref{EqI}-\eqref{EqII} 
can model, e.g., the spreading of a bacterial population under the influence of chemotaxis. Chemotaxis systems have been much studied  in the recent decades. We refer the interested reader to surveys \cite{Horstmann,BBTW} which cover both modelling and analytical aspects. The available results mostly focus on the uniform boundedness/blow-up for finite/infinite times and convergence of solutions to an   equilibrium. It is usually assumed that $\alpha=0$, $\gamma=1$, and very specific  nonlinearities $f$ and $g$ are chosen. For example, in the case of the classic Keller-Segel model \cite{Horstmann} one has $f\equiv0$, i.e., the absence of proliferation which is not realistic in general. Furthermore, the condition $\alpha=0$ corresponds to the standard non-degenerate diffusion. It is has a well known property of the infinite speed of propagation which entails that the population fills the domain instantaneously. Particularly in the case of a bacterial biofilm this falls short to model the experimentally { and numerically observed \cite{eberl2017spatially}} moving fronts. Thus, it was proposed in \cite{ESe} to consider rather general nonlinearities $f$ and $g$, thus allowing to model reaction/interaction, and to take $\alpha>0$. The latter corresponds to the case of a degenerate diffusion (that is, the diffusion coefficient has at least one zero point) of the porous medium type. It is well known (see, e.g., \cite{Vazquez}) that such diffusion leads to solutions with a finite speed of propagation. 

From the analytical point of view, system \eqref{PMCT} is a blend of a porous medium equation with a chemotaxis growth system. The dynamics of a single equation with a porous medium degeneracy has been thoroughly studied in \cite[Chapter 4]{bookMe} with the help of exponential  attractors (see {\it Definition \ref{Definition III.1}} below). For non-degenerate chemotaxis growth systems under the homogeneous Neumann boundary conditions the existence of  attractors and  their dimension  were studied in \cite{EfNakOs2008,NakOsa2013,EfNakWen1,EfNakWen2,EfNakWen3,AidaTsuYagi1,EfYagi1}, see also \cite[Section 3.6]{bookMe}. 
For system \eqref{PMCT} the well-posedness and the existence of the global attractor were established in \cite{ESe,EZh} and \cite{ESZ,EZ2,EZDE1D},  respectively, (see also \cite[Chapters 9-10]{bookMe} and \cite{DisserAZ}). The question of finiteness of the attractor dimension has not as yet  been   studied. We address this issue in the present paper. The answer turns out to be positive under suitable conditions on the problem coefficients. { Thus, to the best of our knowledge,  system \eqref{PMCT} is the second after the porous medium equation}  class of highly degenerate problems which can exhibit finite-dimensional dynamics.

It is well known that the concept of  global attractor has some essential drawbacks. It is in general not stable under perturbations, the speed of convergence to it may be arbitrary slow, and it is usually hard to express it in terms of the parameters of the system. Thus it is often difficult to observe the global attractor in numerical simulations. The notion of exponential attractor (compare {\it Definition \ref{Definition III.1}} below) was first introduced in \cite{EFNT} as an alternative way to capture the dynamics of a dynamical system. It is a finite dimensional positively invariant attracting set that attracts bounded subsets of the phase space with exponential speed.
If such a set exists for a dynamical system, it necessarily contains the global attractor of the system, and that global attractor has finite dimension (so it is also a way to show that the global attractor is finite-dimensional). While generally stable and easier to handle, this (eventually bigger) attracting set has it's own faulty: it is not uniquely determined (while the global attractor is). 

Unlike the nondegenerate dissipative equations and systems on bounded domains which, as a rule, possess  finite-dimensional global and exponential attractors, the dynamics of degenerate problems is much more delicate.  
The porous medium and  $p$-Laplace equations are two very first examples  of autonomous equations which -under rather general conditions- have infinite dimensional
attractors, see \cite{EfZelik,EfOt1,EfOt2}, also \cite[Chapters 4-7]{bookMe}. Moreover, the asymptotics of their  Kolmogorov $\varepsilon$-entropy turned out to be  polynomial. Even the attractors of nondegenerate problems in unbounded domains, which are known  to be of infinite dimension, always showed only logarithmic asymptotics of the entropy, see, e.g.,  \cite{EfFInIfin}. 
%

Previous results  on the well-posedness and the existence of the  global attractor were obtained for \eqref{PMCT} under the balance conditions 
\begin{align}
	\frac{\alpha}{2}+1\leq\gamma<\alpha\label{bc}
\end{align}
and the following assumptions on the nonlinearities $f$ and $g$: 
\begin{align}
	&|f(M,\rho)|\leq \Cl[F]{F_1}(1+|M|^\xi)^{\frac{1}{2}} \quad \text{ for all }M,\rho\geq0\quad\text{for some } \xi\in \left[0,\alpha-\gamma+2\right),\ \Cr{F_1}\in\R^+_0,\label{absf}\\
	&f(M,\rho)\geq\Cl[F]{F_2} M-\Cl[F]{F_3}\quad \text{ for all }M,\rho\geq0\quad\text{for some } \Cr{F_2}\in\R^+,\ \Cr{F_3}\in\R^+_0,\label{fM}\\
	&g(M,\rho)=\Cl[G]{G_1}\rho+g_2(\rho)M,\ |g_2(\rho)|\leq \Cl[G]{G_2}\quad \text{ for all }M,\rho\geq0\quad\text{for some } \Cr{G_1},\Cr{G_2}\in\R_0^+\label{g22},\\
&f(M,\rho)=\Cl[F]{F_4}M+\widetilde{f}\left(M^{1+\frac{\alpha}{2}},\rho\right)\quad \text{ for all }M,\rho\geq0\quad\text{for some }\Cr{F_4}\in\R,\label{tildef}\\
&\widetilde{f}\in W^{1,\infty}_{loc}(\R^2),\ g_2\in W^{1,\infty}_{loc} (\R),\,f(0,\rho)=0\quad\text{for all }\rho\in\R,\quad g_2(0)\leq 0.\label{g2}	
\end{align}
In this setting, we  established in \cite{ESZ}  the existence of the weak global attractor in the phase space $\LInf\times W^{1,\infty}(\Omega)$. Note that in \cite[Section 4.4]{bookMe} it was shown that the dimension of the global attractor for the porous medium equation (thus, even without chemotaxis) 
\begin{align}
	{\partial_tM}= \nabla\cdot\left(M^\alpha \nabla M\right)-f(M)\label{PME}
\end{align}
may be infinite if $-f'(0)>0$. Observe that this includes the case of the standard logistic growth, i.e., when
\begin{align}
 -f(M)=rM\left(1-\frac{M}{K}\right)
\end{align}
for some growth rate $r>0$ and carrying capacity $K>0$. 
Conditions \eqref{bc}-\eqref{g2}   therefore  cannot {\cb guarantee}
 the finite dimension of the global attractor (and hence also the existence of an exponential attractor)  for \eqref{PMCT} { as the following example illustrates.
\begin{Example}
Let 
\begin{align*}
 &-f=-f(M)=M,\qquad g\equiv 0.
\end{align*}
Observe that this choice of $f$ and $g$ is in line with conditions \eqref{bc}-\eqref{g2}. Let us assume further that 
\begin{align*}
\rho_0\equiv 1.
\end{align*}
In this special case equation  \eqref{EqII} together with the corresponding boundary condition can be easily solved explicitly, the solution being $\rho\equiv1$. Hence, the taxis flux in \eqref{EqI} completely vanishes on $\Omega$. As a result, $M$ solves the porous medium equation \eqref{PME} with $-f'(0)=1>0$ thus leading to the infinite dimensional global attractor already for $M$-component.  
\end{Example}
}
We improve  conditions  \eqref{bc}-\eqref{g2} in the following way: we consider now sharper balance conditions than (\ref{bc}), namely
\begin{align}
	1+\frac{\alpha}{2}<\gamma<\alpha\label{newbc}
\end{align}
and replace \eqref{tildef} by 
\begin{align}
&f(M,\rho)=\Cl[F]{F_5}M+\widetilde{f}\left(M^{\beta},\rho\right),\quad\text{for some } \beta>1+\frac{\alpha}{2},\ \Cr{F_5}>0.\label{tildefnew}
\end{align}
The following choice of functions $f$ and $g$ satisfies  conditions \eqref{absf}-\eqref{g22}, \eqref{g2} and \eqref{tildefnew}:
\begin{Example}
\begin{align}	f(M,\rho)&=-M+\frac{M^{\beta}}{M^{\beta}+1}{\arctan\rho},\nonumber\\
g(M,\rho)&=\rho+M\frac{\rho}{\rho+1}.\nonumber
\end{align}
\end{Example}\noindent
We recall a definition of the exponential attractor: 
\begin{Definition}[{\cite[Chapter 3, Definition 3.1]{bookMe}}]\label{Definition III.1}A set ${\cal M}$ is an
exponential attractor\index{exponential attractor} for a semigroup $S(t)$ in a Banach space $X$ if: it
\begin{enumerate}[(i)]
	\item is compact in $X$;
	\item is positively invariant, i.e., $S(t){\cal
M}\subset {\cal M},\ \forall t\ge 0$;
	\item attracts bounded sets of
initial data exponentially fast in the following sense: there exists a monotonic
function $Q$ and a constant $C_{rate}>0$ (called below attraction parameters) such that
$$\forall B\subset X\ \text{bounded},\ \dist_{X}(S(t)B,{\cal M})\le
Q({{\Vert B\Vert }_X}){e^{-C_{rate} t}},\ t\ge 0;  \label{3.4}
$$
\item has finite fractal dimension.
\end{enumerate}
\end{Definition}
Here $\dist_X(\cdot,\cdot)$ denotes the  nonsymmetric Hausdorff distance
between subsets of $X$:
\begin{align*}
 \dist_X(A,B):=\underset{x\in A}{\sup}\underset{y\in B}{\inf}||x-y||_X\text{ for all }A,B\subset X.
\end{align*}
Our main result deals with the existence of exponential attractors for system  \eqref{PMCT}. It reads:
\begin{Theorem}[Exponential attractor for \eqref{PMCT}]\label{mainexpo} Let $\Omega$ be a smooth bounded domain in $\R^N$, $N\in\{1,2,3\}$.
Let the functions $f$ and $g$ satisfy the assumptions \eqref{absf}-\eqref{g22}, \eqref{g2} and \eqref{tildefnew}
and let the given constants $\alpha$ and $\gamma$ satisfy
$1+\frac{\alpha}{2}<\gamma<\alpha$. Then the initial boundary value problem \eqref{PMCT} generates a  well defined semigroup $S(t)$, $t\geq0$, in the (positive cone of the) space $\X$. The semigroup $S(t)$ possess an exponential attractor ${\cal M}$ (in terms of {\it Definition \ref{Definition III.1}}) which is a bounded subset of $C^{\theta}(\overline{\Omega})\times C^{2+\theta}(\overline{\Omega})$ for some H\"older exponent $\theta\in(0,1)$. The number $\theta$ and such parameters of the attractor as: its diameter, fractional dimension, the attraction parameters  $C_{rate}$ and  $Q$ can be chosen such as to depend  only upon the parameters of the problem.
\end{Theorem}\noindent
As a direct corollary of {\it {\cb Theorem \ref{mainexpo}}} we have that
\begin{Corollary}[Finite-dimensional global attractor for \eqref{PMCT}]
 Under assumptions of {\it {\cb Theorem \ref{mainexpo}}} the semigroup generated by system \eqref{PMCT} possesses the finite-dimensional global attractor ${\cal A}\subset {\cal M}$. In particular, upper bounds for the attractor diameter and fractal dimension can be chosen to depend upon the parameters of the problem alone. 
\end{Corollary}

There are different constructions \cite{bookMe} of exponential attractors. We use the one based on the so-called smoothing property { (see \cite{bookMe} and references therein).} 
One of its simplest abstract versions insures the existence of an exponential attractor for a discrete semigroup $S^n$, $n\in\N_0$, and takes (see, e.g., \cite[Chapter 3]{bookMe}) the form 
\begin{align}
 \|S(u_1)-S(u_2)\|_{{\cb H_1}}\leq K \|u_1-u_2\|_{H}\text{ for all }u_1,u_2\in C.\label{basic}
\end{align}
Here $H$ and $H_1$ are two Banach spaces such that $H_1$ is compactly embedded
in $H$, $S$ maps between $H$ and  $H_1$, $C$ is a subset of some metric space $X$ and is invariant under $S$, and the constant $K\geq0$ is independent of a particular choice of $u_1$ and $u_2$. It is in general not difficult to establish such a property for the semigroup corresponding to a dissipative nondegenerate problem. Moreover, in these cases the ways to choose  spaces $H$ and ${\cb H_1}$ in an appropriate way are usually in abundance. Very often they are two H\"older spaces or, alternatively, a Lebesgue and a Sobolev space, defined for the whole spatial domain. However, a condition like \eqref{basic}  is in general unattainable  for a semigroup of solution operators for a degenerate equation, such  {\cb as} e.g. the porous medium equation. In \cite{EfZelik,bookMe} the smoothing property could be  generalised to a form that allows to treat the latter case. It turned out that the underlying spaces, such as $H$ and $H_1$, cannot be chosen once and for all, but that they need to be changed as one passes from a neighbourhood of one point $u_0$ in $C$ to another. Thereby, it is necessary to work on functional spaces set up not on the whole spatial domain, but, rather, on some sub- and superlevel sets of $u_0$. This requires localising techniques. In the present work we use the ideas which were originally developed in  \cite{EfZelik,bookMe} for the porous medium equation in order to obtain an exponential attractor  for system \eqref{PMCT}.  The presence of a chemotaxis transport term in addition to a degenerate diffusion is a considerable complication. It further reduces the class  of norms in which one can estimate the differences of two solutions. For example, while the solution operators of the porous medium equation are Lipschitz continuous both in $L^1$ (this was essentially used in \cite{EfZelik,bookMe}) and $H^{-1}$, in our case they are Lipschitz  only in $H^{-1}$. Working in negative Sobolev spaces is more difficult since they are much less suited for the localising techniques.

The rest of the paper is organised as follows. In {\it Section \ref{Prelims}} we fix some notation and then establish some results concerning the regularity and stability of solutions, as well as some properties of an exponentially absorbing set for system \eqref{PMCT}. In {\it Section \ref{secsmooth}} we formulate and prove a smoothing property ({\it Theorem \ref{theoSP}} below) for the corresponding semigroup. 
The proof of {\it Theorem \ref{mainexpo}} is given in {\it Section~\ref{mainproof}}.
\section{Preliminaries}\label{Prelims}
In this Section we collect some necessary preliminary observations and results. 
\subsection*{Basic notation and functional spaces}\label{not}
We denote $\R^+:=(0,\infty)$, $\R^+_0:=[0,\infty)$. 

Partial derivatives in the  classical or distributional sense with respect to a variable  $z$ are denoted by $\partial_{z}$.   Further, $\nabla$  and $\Delta$ stand for the spatial gradient and Laplace operator{\cb s}, respectively.

We assume the reader to be familiar with the standard $L^p$, Sobolev, and H\"older spaces and their usual properties, as well as with the  more general $L^p$ spaces of functions with values in general Banach spaces, and with anisotropic spaces, such as, for any open $O\subset\R^N$ and $0<t_1<t$, the  parabolic spaces
\begin{align}
&H^1( (t_1,t),H^1_0(O),H^{-1}(O)):=\left\{u\in L^2((t_1,t),H^1_0(O))|\ \partial_t u\in L^2((t_1,t),H^{-1}(O))\right\}\nonumber
\end{align}
equipped with the norm
\begin{align}
&\|u\|_{H^1((t_1,t),H^1_0(O),H^{-1}(O))}:=\left(\|u\|_{L^2((t_1,t),H^1_0(O))}^2+\|\partial_t u\|_{L^2((t_1,t),H^{-1}(O))}^2\right)^{\frac{1}{2}}\nonumber
\end{align}
and
\begin{align}
 &W^{(1,2),2}((t_1,t)\times O):=\left\{u\in L^2((t_1,t),H^2(O))|\ \partial_t u\in L^2((t_1,t),L^2(O))\right\}\nonumber
 \end{align}
equipped with the norm
 \begin{align}
 &\|u\|_{W^{(1,2),2}((t_1,t)\times O)}:=\left(\|u\|_{L^2((t_1,t),H^2(O))}^2+\|\partial_t u\|_{L^2((t_1,t),L^2(O))}^2\right)^{\frac{1}{2}}.\nonumber
\end{align}

{ As usual, $C^k(\overline{\Omega})$, $k\in\N_0$, denotes the space of  $k$ {\cb times} continuously differentiable functions $u:\overline{\Omega}\rightarrow\R$, and $D^{\alpha}=D^{\alpha_1}\dots D^{\alpha_n}$, $\alpha_i\in\N_0$, is the corresponding higher order partial derivative of order $|\alpha|=\sum_{i=1}^{N}\alpha_i$. A norm on $C^k(\overline{\Omega})$ is given by
\begin{align}
 \|u\|_{C^k(\overline{\Omega})}:=\underset{|\alpha|\leq k}{\sum}\max_{x\in \overline{\Omega}}|D^{\alpha}u(x)|.\nonumber
\end{align}
} We recall that a H\"older coefficient for a H\"older exponent $\theta\in(0,1)$ and a real-valued function $w$ defined in a set $A\subset \R^k$, $k\in\N$, is given by 
\begin{align*}
 |w|_{C^{\theta}(A)}:=\underset{x,y\in A,\ x\neq y}{\sup}\frac{|w(x)-w(y)|}{|x-y|^{\theta}}.
\end{align*}
{ This allows to introduce the standard H\"older spaces for $k\in\N_0$:
\begin{align}
 C^{k+\theta}(\overline{\Omega}):=\left\{u\in C^k(\overline{\Omega}):\ |D^{\alpha}u|_{C^{\theta}(\overline{\Omega})}<\infty\text{ for all }|\alpha|=k\right\}\nonumber
\end{align}
equipped with the norm
\begin{align}
 \|u\|_{C^{k+\theta}(\overline{\Omega})}:=\|u\|_{C^k(\overline{\Omega})}+\underset{|\alpha|=k}{\sum}|D^{\alpha}u|_{C^{\theta}(\overline{\Omega})}.\nonumber
\end{align}
}

Recall also that due to the Sobolev interpolation inequality for any  $\theta\in(0,1)$ there exist {\cb numbers $\theta_1,\theta_2\in(0,1)$} and $\Cl{Cinterp}>0$ such that following interpolation inequalities hold:
\begin{align}
 &\|w\|_{\LInf}\leq \Cr{Cinterp}\|w\|_{C^{\theta}(\overline{\Omega})}^{1-{\cb \theta_1}}\|w\|_{\HmOne}^{{\cb \theta_1}}\quad \text{for all }w\in C^{\theta}(\overline{\Omega}),\label{SobInter}\\
 &\|v\|_{W^{1,\infty}(\Omega)}\leq \Cr{Cinterp}\|v\|_{C^{2+\theta}(\overline{\Omega})}^{1-{\cb \theta_2}}\|v\|_{\LTwo}^{{\cb \theta_2}}\quad \text{for all }v\in  C^{2+\theta}(\overline{\Omega}).\label{SobInter2}
\end{align} 

Finally, we make the following two useful conventions. Firstly, for all indices $i$,  $C_i$  denotes a positive constant or, alternatively, a positive continuous function. 
Secondly, the statement that a constant depends on the parameters of the problem means that it depends upon such parameters as:  space dimension $N$, domain $\Omega$,  constants $\alpha$, $\beta$, $\gamma$, $F_i$, $G_i$, 
and 
norms of $\widetilde{f}$ and $g_2$. This dependence  is subsequently {\bf not} indicated in an explicit way.

\subsection*{Sub- and superlevel sets}
In what follows we  sometimes consider parts of solutions of problem \eqref{PMCT} restricted to the sublevel sets
$\left\{M_0>\delta\right\}$ for  $\delta\in\left(0,\|M_0\|_{\LInf}\right)$. Observe that if $M_0\in C^{\theta}(\overline{\Omega})$, then, in fact, for all $0<\delta<\frac{1}{2}\|M_0\|_{\LInf}$
\begin{align}	d_{\R}\left(\left\{M_0\leq \delta\right\},\left\{M_0\geq 2\delta\right\}\right)\geq \delta^{\frac{1}{\theta}}|M_0|_{C^{\theta}(\overline{\Omega})}^{-\frac{1}{\theta}},\label{dist}
\end{align}
where $d_{\R}$ denotes the standard metric distance between sets in $\R$:
$$
d_{\R}(X,Y):=\underset{(x,y)\in X\times Y}{\inf}|x-y|.
$$
Thus, we have a control over a lower bound for the distance between sub- and superlevel sets, and that bound depends only upon the quantities which appear on the right-hand side of \eqref{dist}. 
An important consequence of this observation is the existence for all $0<\delta_0<\delta_1<\frac{1}{2}\|M_0\|_{\LInf}$ of a smooth cut-off function $\varphi$ which satisfies the following: 
\begin{subequations}\label{cutoff}
\begin{align}
	&\varphi\in C^{\infty}_0(\Omega),\\
	&\varphi\in [0,1] \text{ in }\overline{\Omega},\quad \varphi=0 \text{ in }\left\{M_0<\delta_0\right\},\quad \varphi=1 \text{ in } \left\{M_0> {\cb \delta_1}\right\},\\
	&\left|D^k\varphi(x)\right|\leq \Cl{Cphi}\varphi^{1-\omega}(x) \text{ for all }x\in \overline{\Omega} \text{ for all } \omega \in (0,1) \text{ and } k\in\N,\label{nablaphi}\\
	&\Cr{Cphi}=\Cr{Cphi}\left(|M_0|_{C^{\theta}(\overline{\Omega})},\delta_0,\delta_1,\theta,\omega,k\right).
\end{align}
\end{subequations}
A family of functions with such properties exists for all $M_0\in C^{\theta}(\overline{\Omega})$ due to  property \eqref{dist}, see Proposition 1.1 of \cite{bookMe}. 
\subsection*{Regularity of solutions}
From now on we assume that assumptions of {\it {\cb Theorem~\ref{mainexpo}}} are fulfilled. 
{\cb A solution to \eqref{PMCT} can be defined as follows:
\begin{Definition}[Weak solution]\label{defsol}
Let $(M_{0},\rho_{0})\in L^{\infty}(\Omega)\times W^{1,\infty}(\Omega)$. 
We call a pair of functions $M,\rho:\R^+_0\times\overline{\Omega}\rightarrow\R^+_0$  a global weak solution of
 \eqref{PMCT} if for all $0<T<\infty$ it holds that
\begin{enumerate}[(i)]
\item $M\in L^{\infty}\left((0,T)\times\Omega\right)$, $M^{\alpha+1} \in
 L^{2}\left((0,T);\HOneO\right)$, $\partial_t{M}\in L^{2}\left((0,T);\HmOne\right)$;
\item $\rho\in L^{\infty}((0,T);W^{1,\infty}(\Omega))$, $\partial_t\rho\in L^{2}\left((0,T);\HmOne\right)$;
\item $(M,\rho)$ satisfies  equations \eqref{EqI}-\eqref{EqII} in $L^{2}\left((0,T);\HmOne\right)$;
\item $(M,\rho)(0)=(M_0,\rho_0)$ in $\HmOne\times L^2(\Omega)$.
\end{enumerate}
\end{Definition}
\noindent It was proved in \cite{ESe, EZh} (see also \cite[Section 3.2]{DisserAZ}) that for all $(M_{0},\rho_{0})\in L^{\infty}(\Omega)\times W^{1,\infty}(\Omega)$ the initial boundary value problem \eqref{PMCT} possess a unique solution with regularity as stated in {\it Definition \ref{defsol}} and, moreover, this solution is  uniformly bounded in $\X$. In general, a  solution of a degenerate equation like \eqref{EqI} is only weak and not classical}  \cite{Vazquez}. Still, it is well-understood \cite{AroBen1979,DiBen83,CaffFri1980,Ivanov86,Ziemer1982} that under reasonable conditions on the equation coefficients bounded weak solutions are H\"older continuous. In our case the following regularity result holds:
\begin{Lemma}[Regularity and positivity]\label{LemmaHoel}
 Let $(M_0,\rho_0)\in\X$ with $\|(M_0,\rho_0)\|_{\X}\leq R$ for some $R>0$ and let $(M,\rho):[0,T]\times\overline{\Omega}\rightarrow\R^+_0\times\R^+_0$ be the corresponding weak solution  to \eqref{PMCT}. Then:
 \begin{enumerate}
  \item(H\"older regularity)  There exists a number $\theta=\theta(R)\in(0,1)$   such that $(M,\rho)$ belongs to $C^{\frac{\theta}{2},\theta}(\R^+\times \overline{\Omega})\times C^{1+\frac{\theta}{2},2+\theta}(\R^+\times \overline{\Omega})$, and for all $0<\tau<T$ it holds that
 \begin{align}
 \|(M,\rho)\|_{C^{\frac{\theta}{2},\theta}([\tau,T]\times \overline{\Omega})\times C^{1+\frac{\theta}{2},2+\theta}([\tau,T]\times \overline{\Omega})}\leq \C\left(\tau,T,R\right).\label{HoelMrho}
\end{align}
\item(Preservation of positivity) $\R^+\times\{M_0>0\}\subset\{M>0\}$, and for all $\delta\in \left(0,\|M_0\|_{\LInf}\right)$ and $T>0$ it holds that
\begin{align}
 \inf\left\{M(s,x)|\ t\in[0,T],\,x\in\left\{M_0>\delta\right\}\right\}\geq \Cl{C26}(\delta,T,R).\label{MPos1}
\end{align}
\item(Regularity on sublevel sets) $M\in C^{1,2}\left(\R^+\times{\cb \overline{\{M_0>\delta\}}}\right)$, and for all $\delta\in \left(0,\|M_0\|_{\LInf}\right)$ and $0<\tau<T$ it holds that
\begin{align}
 \|M\|_{C^{1,2}\left([\tau,T]\times\overline{\{M_0>\delta\}}\right)}\leq \C\left(\delta,\tau,T,R\right).\label{MC12}
\end{align}
 \end{enumerate}
\end{Lemma}
\begin{proof}
 {\bf 1.} 
 Observe that equation \eqref{EqI} can be written in the following form:
\begin{align*}
 &\partial_t M=\nabla\cdot { A}(t,x,M,\nabla M)+b(t,x),
\end{align*}
where we introduced
\begin{align*}
 &{ A}(t,x,M,p):=M^{\alpha}p{\cb-}M^{\gamma}\nabla \rho(t,x),\\
 &b(t,x):=f(M,\rho)(t,x).
\end{align*}
Clearly, functions ${\cb A}$ and $b$ satisfy the following conditions:
\begin{align}
 &|{ A}(t,x,M,p)|\leq M^{\alpha}|p|+\left\|M^{\gamma}\nabla \rho\right\|_{(L^{\infty}(\R^+\times\Omega))^n},\label{abnd1}\\
 &|b(t,x)|\leq \|f(M,\rho)\|_{L^{\infty}(\R^+\times\Omega)}.\label{bbnd}
\end{align}
{ 
Moreover, with the help of the H\"older inequality we deduce that
\begin{align}
 { A}(t,x,M,p)\cdot p=&M^{\alpha}|p|^2+M^{\gamma}\nabla \rho(t,x)\cdot p\nonumber\\
  = &M^{\alpha}|p|^2+M^{\frac{\alpha}{2}}p\cdot M^{\gamma-\frac{\alpha}{2}}\nabla\rho(t,x)\nonumber\\
  \geq &M^{\alpha}|p|^2-M^{\frac{\alpha}{2}}|p|\cdot M^{\gamma-\frac{\alpha}{2}}|\nabla\rho(t,x)|\nonumber\\
 \geq& \frac{1}{2}M^{\alpha}|p|^2-\frac{1}{2}\left\|M^{\gamma-\frac{\alpha}{2}}\nabla\rho\right\|_{(L^{\infty}(\R^+\times\Omega))^n}^2,\label{Adotp}
\end{align}
}Due to  \eqref{abnd1}-\eqref{Adotp}, $\alpha>0$, $\gamma\geq \frac{\alpha}{2}$, and the fact that $(M,\rho)$ is uniformly bounded in $\X$ by a constant which depends only upon the parameters of the problem and $R$, we can apply Theorems 2.I and 3.I from \cite{Ivanov86} on inner and boundary regularity for degenerate parabolic PDEs. These results  {\cb imply the existence of} a number $\theta=\theta\left(R\right)\in(0,1)$  such that $M\in C^{\frac{\theta}{2},\theta}([\tau,T]\times \overline{\Omega})$ for all $0<\tau<T$, and 
\begin{align}
 \|M\|_{C^{\frac{\theta}{2},\theta}([\tau,T]\times \overline{\Omega})}\leq \C\left(\tau,T,R\right).\label{HoelM}
\end{align}
Consequently, equation \eqref{EqII} together with the boundary condition $\rho\equiv 1$ can be seen as a linear parabolic equation for $\rho$ with H\"older continuous coefficients. Thus, standard Schauder estimates entail that $\rho$ is a classical solution to  \eqref{EqII} and satisfies 
\begin{align}
 \|\rho\|_{C^{1+\frac{\theta}{2},2+\theta}([\tau,T]\times \overline{\Omega})}\leq \C\left(\tau,T,R\right).\label{Hoelrho}
\end{align}
Combining \eqref{HoelM}-\eqref{Hoelrho}, we obtain \eqref{HoelMrho}.

{\bf 2.} { We start by proving the quantitative estimate \eqref{MPos1}. For this purpose we make use of  the classical idea of propagation of $L^p$ bounds. Since we aim at an estimate from {\cb below}, we  estimate $M^{-1}$ from above. Of course, this can only be done in {\cb  those} areas where the $M$-component is strictly bounded from below by a positive constant. {\cb  For that reason} we use a cutoff function  {\cb from  \eqref{cutoff}  in order} to eliminate the {\cb part of $\Omega$} where $M_0$ is small. More precisely, let us} multiply equation \eqref{EqI} by $-a\varphi_{\delta}^a M^{-a-1}$ for ${\cb a\geq2\alpha}$ and $\varphi_{\delta}$ as in \eqref{cutoff} {\cb while} we choose $\delta_0:=\delta$, $\delta_1:=2\delta$, so that, in particular, $\varphi_{\delta}=0$ in $\left\{|M_0|\leq\delta\right\}$. Integrating (formally) over $\Omega$ and using {\cb integration by parts} where necessary we obtain that
\begin{align}	
&\frac{d}{dt}\left\|\varphi_{\delta}^{a} M^{-a}\right\|_{\LOne}\nonumber\\
=&-4\frac{(a+1)a}{(a-\alpha)^2}\left\|\varphi_{\delta}^{\frac{a}{2}}\nabla M^{-\frac{a-\alpha}{2}}\right\|_{\LTwon}^2-2a\frac{a}{a-\alpha}\left(\varphi_{\delta}^{\frac{a}{2}}\nabla M^{-\frac{a-\alpha}{2}},\left(\varphi_{\delta}M^{-1}\right)^{\frac{a}{2}-1}M^{\frac{\alpha-2}{2}}\nabla\varphi_{\delta}\right)_{\LTwon}\nonumber\\	
&+2a\frac{a+1}{a-\alpha}\left(\varphi_{\delta}^{\frac{a}{2}}\nabla M^{-\frac{a-\alpha}{2}},\left(\varphi_{\delta}M^{-1}\right)^{\frac{a}{2}}M^{\gamma-1-\frac{\alpha}{2}}\nabla\rho\right)_{\LTwon}\nonumber\\	
&{\cb-}a^2\left(\left(\varphi_{\delta}M^{-1}\right)^{a-1}M^{\gamma-2},\nabla\varphi_{\delta}\cdot\nabla\rho\right)_{\LTwon}{\cb +}a\left(\varphi_{\delta}^{a} M^{-a},\frac{f{\cb(M,\rho)}}{M}\right)_{\LTwo}.\label{S1}	
\end{align}	
Using the Young and H\"older inequalities,  the assumptions on $\alpha$, $\gamma$, and $f$, and the properties of $\varphi_{\delta}$ we estimate the right-hand side of \eqref{S1} on $[0,T]$ as follows:
\begin{align}	
&\frac{d}{dt}\left\|\varphi_{\delta}^{a} M^{-a}\right\|_{\LOne}\nonumber\\
\leq&-\C\left\|\varphi_{\delta}^{\frac{a}{2}}\nabla M^{-\frac{a-\alpha}{2}}\right\|_{\LTwon}^2+\C(\delta,T,R) a\left\|\varphi_{\delta}^{\frac{a}{2}}\nabla M^{-\frac{a-\alpha}{2}}\right\|_{\LTwon}\left\|\varphi_{\delta}^{a} M^{-a}\right\|_{\LOne}^{\frac{a-2}{2a}}\nonumber\\	
&+\C(\delta,T,R) a\left\|\varphi_{\delta}^{\frac{a}{2}}\nabla M^{-\frac{a-\alpha}{2}}\right\|_{\LTwon}\left\|\varphi_{\delta}^{a} M^{-a}\right\|_{\LOne}^{\frac{1}{2}}+\C(\delta,T,R)a^2\left\|\varphi_{\delta}^{a} M^{-a}\right\|_{\LOne}^{\frac{a-1}{a}}\nonumber\\
&+\C(\delta,T,R)a\left\|\varphi_{\delta}^{a} M^{-a}\right\|_{\LOne}\nonumber\\
\leq& -\Cl{_C23}(\delta,T,R)\left\|\varphi_{\delta}^{\frac{a}{2}}\nabla M^{-\frac{a-\alpha}{2}}\right\|_{\LTwon}^2+\Cl{_SC1}(\delta,T,R)a^2\left(\left\|\varphi_{\delta}^{a} M^{-a}\right\|_{\LOne}+1\right)\nonumber\\
\leq& -\Cr{_C23}(\delta,T,R)\left(\frac{1}{2}\left\|\nabla\left(\varphi_{\delta}^{\frac{a}{2}} M^{-\frac{a-\alpha}{2}}\right)\right\|_{\LTwon}^2-\frac{a^2}{4}\left\|\left(\varphi_{\delta}M^{-1}\right)^{\frac{a}{2}-1}M^{\frac{\alpha-2}{2}}\nabla\varphi\right\|_{\LTwon}^2\right)\nonumber\\
&+\Cr{_SC1}(\delta,T,R)a^2\left(\left\|\varphi_{\delta}^{a} M^{-a}\right\|_{\LOne}+1\right)\nonumber\\
\leq&-\Cl{C23}(\delta,T,R)\left\|\nabla\left(\varphi_{\delta}^{\frac{a}{2}} M^{-\frac{a-\alpha}{2}}\right)\right\|_{\LTwon}^2+ \C(\delta,T,R)a^2\left\|\varphi_{\delta}^{a} M^{-a}\right\|_{\LOne}^{\frac{a-2}{a}}\nonumber\\
&+\Cr{_SC1}(\delta,T,R)a^2\left(\left\|\varphi_{\delta}^{a} M^{-a}\right\|_{\LOne}+1\right)\nonumber\\
\leq&-\Cr{C23}(\delta,T,R)\left\|\nabla\left(\varphi_{\delta}^{\frac{a}{2}} M^{-\frac{a-\alpha}{2}}\right)\right\|_{\LTwon}^2+ \Cl{SC1}(\delta,T,R)a^2\left(\left\|\varphi_{\delta}^{a} M^{-a}\right\|_{\LOne}+1\right).\label{estMm}
\end{align}
The first consequence of \eqref{estMm} is due to the Gronwall lemma:	
\begin{align}	
    \underset{t\in[0,T]}{\sup}\left\|\varphi_{\delta}^{a}M^{-a}(t)\right\|_{\LOne}+1\leq & e^{T\Cr{SC1}(\delta,T,R)a^2}\left(\left\|\varphi_{\delta}^{a}M_0^{-a}\right\|_{\LOne}+1\right)\nonumber\\	
    \leq & e^{T\Cr{SC1}(\delta,T,R)a^2}\left(\delta^{-a}+1\right)\nonumber\\	
    =:&\Cl{Ba}(a,\delta,T,R).\label{Mma}	
\end{align}	
This shows the a priori boundedness of $\left\|\varphi_{\delta}M^{-1}\right\|_{\cb L^a(\Omega)}$ for all $a\in{\cb[2\alpha},\infty)$. To get an estimate for $\left\|\varphi_{\delta}M^{-1}\right\|_{\cb \LInf}$ observe that due to the  interpolation inequality for Lebesgue spaces
\begin{align}	
    \left\|\varphi_{\delta}^{a}M^{-a}\right\|_{\LOne}=&\left\|\varphi_{\delta}^{\frac{a}{2}}M^{-\frac{a-\alpha}{2}}M^{-\frac{\alpha}{2}}\right\|_{\LTwo}^2\nonumber\\	
    \leq&\left\|M^{-1}\right\|_{L^{3\alpha}\left(\left\{M_0\geq\delta\right\}\right)}^\alpha\left\|\varphi_{\delta}^{\frac{a}{2}}M^{-\frac{a-\alpha}{2}}\right\|_{L^3(\Omega)}^2\nonumber\\    \leq&\left\|\varphi_{\frac{\delta}{2}}M^{-1}\right\|_{L^{3\alpha}(\Omega)}^\alpha\left\|\varphi_{\delta}^{\frac{a}{2}}M^{-\frac{a-\alpha}{2}}\right\|_{L^3(\Omega)}^2\nonumber\\	
    \leq&\Cr{Ba}^{\alpha}\left(3\alpha,\frac{\delta}{2},T,R\right)\left\|\varphi_{\delta}^{\frac{a}{2}}M^{-\frac{a-\alpha}{2}}\right\|_{L^3(\Omega)}^2.\label{estim}	
\end{align}	
Combining \eqref{estMm} with \eqref{estim}, the Young inequality, and the Sobolev interpolation inequality 
\begin{align}	
        &||u||_{L^3(\Omega)}\leq \Cl{C3}||u||_{\HOneO}^{\frac{4N}{6+3N}}||u||_{\LOne}^{\frac{6-N}{6+3N}}\nonumber	
\end{align}	
 for $u:=\varphi_{\delta}^{\frac{a}{2}}M^{-\frac{a-\alpha}{2}}$ yields	
\begin{align}	
    \frac{d}{dt}\left\|\varphi_{\delta}^{a} M^{-a}\right\|_{\LOne}	
    \leq &-\Cl{SC0_}(\delta,T,R)\left\|\nabla\left(\varphi_{\delta}^{\frac{a}{2}}M^{-\frac{a-\alpha}{2}}\right)\right\|_{\LTwon}^2\nonumber\\ &+\Cl{Ba_1}(\delta,T,R)a^2\left\|\nabla\left(\varphi_{\delta}^{\frac{a}{2}}M^{-\frac{a-\alpha}{2}}\right)\right\|_{\LTwon}^{2\frac{4N}{6+3N}}\left\|\varphi_{\delta}^{\frac{a}{2}}M^{-\frac{a-\alpha}{2}}\right\|_{\LOne}^{2\frac{6-N}{6+3N}}+\Cr{SC1}(\delta,T,R)a^2\nonumber\\    \leq&\Cl{Ba_2}(\delta,T,R)a^{\frac{12+6N}{6-N}}\left\|\varphi_{\delta}^{\frac{a}{2}}M^{-\frac{a-\alpha}{2}}\right\|_{\LOne}^2+\Cr{SC1}(\delta,T,R)a^2\nonumber\\	
\leq&\C(\delta,T,R)a^{\frac{12+6N}{6-N}}\left\|\varphi_{\delta}^{\frac{a}{2}}M^{-\frac{a}{2}}\right\|_{\LOne}^2+\Cr{SC1}(\delta,T,R)a^2\nonumber\\
\leq&\Cl{Ba_3}(\delta,T,R)a^{\frac{12+6N}{6-N}}\left(\left\|\varphi_{\delta}^{\frac{a}{2}}M^{-\frac{a}{2}}\right\|_{\LOne}^2+1\right).\label{estMm_}	
\end{align}	
Integrating \eqref{estMm_} over $[0,T]$ and taking maximum on both sides of \eqref{estMm_1} we obtain  that
\begin{align}	
&\underset{t\in[0,T]}{\max}\left(\left\|\varphi_{\delta}^{a} M^{-a}(t)\right\|_{\LOne}+1\right)\nonumber\\
\leq &	
\left(\left\|\varphi_{\delta}^{a} M_0^{-a}\right\|_{\LOne}+1\right) +\Cr{Ba_3}(\delta,T,R)a^{\frac{12+6N}{6-N}}\int_0^{T}\left\|\varphi_{\delta}^{\frac{a}{2}}M^{-\frac{a}{2}}\right\|_{\LOne}^2+1\,d s\nonumber\\	
\leq &	
    {\cb \left(\delta^{-a}|\Omega|+1\right)} +\Cl{C102}(\delta,T,R)a^{\frac{12+6N}{6-N}}\underset{t\in[0,T]}{\max}\left(\left\|\varphi_{\delta}^{\frac{a}{2}}M^{-\frac{a}{2}}\right\|_{\LOne}+1\right){\cb ^2.}	
    \label{estMm_1}	
\end{align}	
{\cb  Here we used the properties of the cutoff function $\varphi_{\delta}$. Since for $a\geq2\alpha>1$ it holds that 
$\delta^{-a}|\Omega|+1\leq
   \Cl{C100}^a(\delta)$, 
estimate (\ref{estMm_1}) leads 
to a recursive inequality
\begin{align}
 &A_a\leq \Cl{C103}(\delta,T,R)a^{\Cl{C101}}A_{\frac{a}{2}}^2\label{estMm_13}
\end{align}
for 
\begin{align}
 A_a:=\frac{\underset{s\in[0,t]}{\max}\left(\left\|\varphi_{\delta}^{a} M^{-a}(t)\right\|_{\LOne}+1\right)}{\Cr{C100}^a(\delta)}+1.\nonumber
\end{align}
A standard induction argument
together with estimate \eqref{Mma} for $a:=2\alpha$ 
implies that
\begin{align} A_{2^{n+1}\alpha}^{\frac{1}{2^n}}
\leq&\left(\Cr{C103}(\delta,T,R)\alpha^{\Cr{C101}}\right)^{\sum^{n-1}_{k=0}2^{k-n}}2^{\Cr{C101}\sum^{n-1}_{k=0}(n+1-k)2^{k-n}}A_{2\alpha}\nonumber\\
 \underset{n\rightarrow\infty}{\rightarrow}&\Cr{C103}(\delta,T,R)\alpha^{\Cr{C101}}2^{3\Cr{C101}}A_{2\alpha}\nonumber\\
  \leq & \Cl{Binf_1}(\delta,T,R).\label{estMn_4}
\end{align}
Thanks to \eqref{estMn_4} we obtain that
\begin{align}	
    \underset{s\in[0,t]}{\max}\left\|M^{-1}(s)\right\|_{L^{\infty}(\left\{M_0\geq\delta\right\})}	
    \leq &\underset{s\in[0,t]}{\max}\left\|\varphi_{\delta} M^{-1}(s)\right\|_{\LInf}\nonumber\\
    = &\underset{s\in[0,t]}{\max}\underset{n\rightarrow\infty}{\lim}\,\left\|\varphi_{\delta} M^{-1}(t)\right\|_{L^{2^{n+1}\alpha}(\Omega)}\nonumber\\
    \leq &\underset{n\rightarrow\infty}{\lim\sup}\,\Cr{C100}(\delta)A_{2^{n+1}\alpha}^{\frac{1}{2^{n+1}\alpha}}\nonumber\\	
    \leq&\Cr{C100}(\delta)\Cr{Binf_1}^{\frac{1}{2\alpha}}(\delta,T,R)\nonumber\\
    =: &\Cr{Binf}(\delta,T,R).\nonumber
\end{align}	
Consequently, it holds that} 	
\begin{align}	
    \inf\left\{M(s,x)|\ t\in[0,T],\,x\in\left\{M_0>\delta\right\}\right\}\geq \left(\Cl{Binf}\left(\delta/2,T,R\right)\right)^{-1}=:\Cr{C26}(\delta,T,R)>0 \nonumber	
\end{align}
for all $\delta\in \left(0,\|M_0\|_{\LInf}\right)$ and $T>0$, which proves \eqref{MPos1}. 

{\cb The above calculations -and thus also the quantitative estimate \eqref{MPos1}- are only valid if $M(t,\cdot)$ is uniformly bounded from below on $\left\{M_0>\delta\right\}$ for all $t\in[0,T]$. In order to justify the latter we apply a standard approximation argument. In \cite{ESe, EZh} (see also \cite[Section 3.2]{DisserAZ}) we proved that on the one hand a weak solution is unique, and on the other hand it can be obtained as a limit of approximations which solve a nondegenerate system such as, e.g., for  $n\in\N$ the system
\begin{alignat}{3}
&\partial_t{M_n}&=&\nabla \cdot \left(\left(M_n+\frac{1}{n}\right)^{\alpha}\nabla M_n\right)-\nabla \cdot \left(\left(M_n+\frac{1}{n}\right)^\gamma\nabla\rho_n\right)-f\left(M_n,\rho_n\right)& \text{ in } (0,T)\times\Omega,\nonumber\\
&\partial_t{\rho_n}&=&\Delta\rho_n-g(M_n,\rho_n)& \text{ in }(0,T)\times\Omega\nonumber
\end{alignat}
equipped with the same initial and boundary conditions as for the original system \eqref{PMCT}. Since  $M_n\geq0$ solves a nondegenerate equation, its strict positivity is guaranteed. Moreover, it is not difficult to see that family $M_n$, $n\in\N$,  satisfies a positive bound similar to \eqref{MPos1} with some constant  which is independent of $n$. Consequently, the limit function $M$ is indeed uniformly bounded from below by a positive constant and satisfies \eqref{MPos1}.
}

{\bf 3.} Let $M:=u^{\frac{1}{\alpha+1}}$. Under this change of variables equation \eqref{EqI} takes the form
\begin{align}
 \frac{1}{\alpha+1}u^{\frac{1}{\alpha+1}-1}\partial_t u=\frac{1}{\alpha+1}\Delta u-\frac{\gamma}{\alpha+1}u^{\frac{\gamma}{\alpha+1}-1}\nabla\rho\cdot\nabla u-u^{\frac{\gamma}{\alpha+1}}\Delta\rho-f\left(u^{\frac{1}{\alpha+1}},\rho\right).\label{Equ1}
\end{align}
Dividing \eqref{Equ1} by $\frac{1}{\alpha+1}u^{\frac{1}{\alpha+1}-1}$, we obtain that
\begin{align}
 \partial_t u=u^{\frac{\alpha}{\alpha+1}}\Delta u-\gamma u^{\frac{\gamma-1}{\alpha+1}}\nabla\rho\cdot\nabla u-(\alpha+1)u^{\frac{\alpha+\gamma}{\alpha+1}}\Delta\rho-(\alpha+1)u^{\frac{\alpha}{\alpha+1}}f\left(u^{\frac{1}{\alpha+1}},\rho\right),\label{Equ2}
\end{align}
i.e., 
\begin{align}
 \partial_t u=a_0\Delta u+a_1\cdot\nabla u+a_3,\label{Equ3}
\end{align}
with 
\begin{align*}
 &a_0:=u^{\frac{\alpha}{\alpha+1}},\\
 &a_1:=-\gamma u^{\frac{\gamma-1}{\alpha+1}}\nabla\rho,\\
 &a_2:=-(\alpha+1)u^{\frac{\alpha+\gamma}{\alpha+1}}\Delta\rho-(\alpha+1)u^{\frac{\alpha}{\alpha+1}}f\left(u^{\frac{1}{\alpha+1}},\rho\right).
\end{align*}
Due to the results of parts 1. and 2. of this Lemma and assumptions on $\alpha$, $\gamma$, and $f$, we have for all $0<\tau<T$ and $\delta\in(0,\|M_0\|_{\LInf})$ that in $[\tau,T]\times \{M_0>\delta\}=[\tau,T]\times \{u_0>\delta^{\alpha+1}\}$ equation \eqref{Equ3} is  a nondegenerate linear parabolic PDE with H\"older continuous coefficients. Standard result \cite[Chapter 10, Theorem 10.1]{LSU} on {\it interior} regularity in H\"older spaces yields that
\begin{align}
 \|u\|_{C^{1,2}\left([\tau,T]\times\overline{\{u_0>\delta^{\alpha+1}\}}\right)}\leq \C\left(\delta,\tau,T,R\right).\label{uC12}
\end{align}
Since the map $u\mapsto u^{\frac{1}{\alpha+1}}$ is smooth in $\R^+$, \eqref{MC12} is a consequence of \eqref{uC12}. {\it Lemma \ref{LemmaHoel}} is proved.
\end{proof}
\subsection*{Stability}
As was mentioned earlier, the initial boundary value problem \eqref{PMCT} is  well-posed (in the usual Hadamard sense). In particular, the following stability result was proved in {\cb\cite{ESe, EZh}}:
\begin{Lemma}[Lipschitz property, {\cb\cite{ESe,EZh}}]\label{Lipstab}
There exists a function $L_0:\R^+_0\times\R^+_0 \rightarrow\R^+_0$ with $L_0(0,\cdot)\equiv1$, which is continuous, increasing in each variable, depends only upon the parameters of the problem, 
and such that the following Lipschitz-type estimate holds:
\begin{align}	
&\underset{s\in[0,t]}{\max}\|(M_1-M_2,\rho_1-\rho_2)(s)\|_{H^{-1}(\Omega)\times\LTwo}\nonumber\\
&+\left(\int_0^t\left(M_{1}^{\alpha+1}-M_{2}^{\alpha+1}, M_1-M_2\right)_{\LTwo}\,ds\right)^{\frac{1}{2}}+\|{\cb \rho_1-\rho_2}\|_{L^2(0,t;\HOneO)}\nonumber\\
\leq & L_0(t,R)\|(M_{10}-M_{20},\rho_{10}-\rho_{20})\|_{H^{-1}(\Omega)\times\LTwo}.\label{Lip}
\end{align}
\end{Lemma}

It is possible to estimate the difference $M_1-M_2$ in $\LInf$ by replacing \eqref{Lip} with a H\"older property:
\begin{Lemma}[H\"older property]
 There exists a constant $\theta_{\infty}\in(0,1)$ and a {\cb function} $L_1:\R^+_0\times\R^+_0\rightarrow\R^+$  which is continuous, increasing in each variable, depends only upon the parameters of the problem, 
and such that for all $(M_{10},\rho_{10}),(M_{20},\rho_{20})\in C^{\theta}(\overline{\Omega})\times C^{2+\theta}(\overline{\Omega})$ with $\|(M_0,\rho_0)\|_{C^{\theta}(\overline{\Omega})\times C^{2+\theta}(\overline{\Omega})}\leq R$ for some $R\geq0$ it holds for all $t>0$ that
\begin{align}
&\underset{s\in[0,t]}{\max}\|(M_1-M_2,\rho_1-\rho_2)(s)\|_{\LInf}\leq L_1(t,R)\|(M_{10}-M_{20},\rho_{10}-\rho_{20})\|_{H^{-1}(\Omega)\times\LTwo}^{\theta_{\infty}}.\label{Holder}
\end{align}
\end{Lemma}
\begin{proof}
Combining \eqref{Lip} with the interpolation inequalities  \eqref{SobInter}-\eqref{SobInter2} applied to $w=(M_1-M_2)(s)$ and $v=(\rho_1-\rho_2)(s)$, 
 we deduce that {\cb
\begin{align}
&\underset{s\in[0,t]}{\max}\|(M_1-M_2,\rho_1-\rho_2)(s)\|_{\LInf\times W^{1,\infty}(\Omega)}\nonumber\\
\leq & \Cl{BH} L_0^{\theta_1}(t,R)(2R)^{1-\theta_1}\|M_{10}-M_{20}\|_{H^{-1}(\Omega)}^{\theta_1}+\Cr{BH} L_0^{\theta_2}(t,R)(2R)^{1-\theta_2}\|\rho_{10}-\rho_{20}\|_{\LTwo}^{\theta_2}\nonumber\\\leq &L_1(t,R)\|(M_{10}-M_{20},\rho_{10}-\rho_{20})\|_{H^{-1}(\Omega)\times\LTwo}^{\theta_{\infty}},
\nonumber
\end{align} 
where  $\theta_{\infty}:=\min\{\theta_1,\theta_2\}$.}
\end{proof}

\subsection*{Absorbing set} 
 It was proved in \cite{ESZ} that the initial boundary value problem \eqref{PMCT} generates a well defined semigroup $S(t)$, $t\geq0$, in the phase space $\X$ which possesses a bounded  exponentially absorbing  positively invariant set ${\cal B}_0$. In particular, the following dissipative estimate holds \cite{ESZ}: 
\begin{align}
\|(M,\rho)(t)\|_{\X}
\leq & C_{\infty}\|(M_0,\rho_0)\|_{\X}^{r_{\infty}} e^{-\omega_{\infty}t}+D_{\infty}\text{ for all } t\geq 0,\label{TheoDE}
\end{align}
where the positive constants $C_{\infty}, r_{\infty}, \omega_{\infty}, D_{\infty}$ depend only upon the parameters of the problem. Due to {\it Lemma \ref{LemmaHoel}}, the absorbing  set can actually be chosen in a nicer space:
\begin{Lemma}[Absorbing set] \label{absorb}
 The semigroup $S(t)$ possesses an exponentially absorbing positively invariant set ${\cal B}\subset C^{\theta}(\overline{\Omega})\times C^{2+\theta}(\overline{\Omega})$ ($\theta\in(0,1)$ as in {\it Lemma \ref{LemmaHoel}}) such that
 for some $R>0$ which depends only upon the parameters of the problem it holds for all $(M_0,\rho_0)\in {\cal B}$ that 
 \begin{align}
  &\|(M_0,\rho_0)\|_{C^{\theta}(\overline{\Omega})\times C^{2+\theta}(\overline{\Omega})}\leq R.\label{smooth}
 \end{align}
 Moreover, the solutions $(M,\rho):=S(\cdot)(M_0,\rho_0)$ belong to $\left(C^{\frac{\theta}{2},\theta}([0,\infty)\times \overline{\Omega})\cap  C^{1,2}\left(\R^+\times\{M_0>0\}\right)\right)\times C^{1+\frac{\theta}{2},2+\theta}([0,\infty)\times \overline{\Omega})$ and satisfy for all $\delta\in \left(0,\|M_0\|_{\LInf}\right)$ and $t>0$ {\cb the inequalities}
 \begin{align}
 &\|(M,\rho)\|_{C^{\frac{\theta}{2},\theta}([0,t]\times \overline{\Omega})\times C^{1+\frac{\theta}{2},2+\theta}([0,t]\times \overline{\Omega})}\leq \C(t),\label{HoelB}\\
 &\inf\left\{M({\cb s},x)|\ s\in[0,t],\,x\in\left\{M_0>\delta\right\}\right\}\geq \Cl{C26_}(\delta,t),\label{MPos1_}\\
 & \|M\|_{C^{1,2}\left([0,t]\times\overline{\{M_0>\delta\}}\right)}\leq \C(\delta,t).\label{MC12_}
\end{align}
\end{Lemma}
\begin{proof}
 Set ${\cal B}:=S(1){\cal B}_0$.  Due to a standard argument, ${\cal B}$ remains an exponentially absorbing  positively invariant set for $S(t)$. On the other hand, {\it Lemma \ref{LemmaHoel}} ensures \eqref{smooth}-\eqref{MC12_} (choose $\tau:=1$ and $T:=t+1$).
\end{proof}

 
\section{A smoothing property}\label{secsmooth}
The aim of this section is to prove that the semigroup $S(t)$ generated by system \eqref{PMCT} is asymptotically smooth. Recall that due to the general theory presented in \cite{bookMe} {\cb (}see also references therein) every asymptotically smooth semigroup possesses an exponential attractor. More notation is needed first. 
%
Let 
\begin{align}
&u:=(M,\rho),\nonumber\\
&X:=H^{-1}(\Omega)\times\LTwo.\label{Xnorm}
\end{align}
Following \cite{EfZelik}, we introduce for any $u_0:=(M_0,\rho_0)\in{\cal B}$ some suitable $u_0$-dependent spaces and an operator. In our case a possible choice is as follows: for $\delta\in(0,\|M_0\|_{L^{\infty}(\Omega)})$ and $0<t_1<t$ set
\begin{align}
&Y_{u_0}^{(\delta)}:=L^2((t_1,t)\times\{M_0>\delta\})\times\left( L^2((t_1,t),H^1(\{M_0>\delta\}))\cap L^2((t_1,t)\times\Omega)\right),\nonumber\\
&Z_{u_0}^{(\delta)}:=W^{(1,2),2}((t_1,t)\times\{M_0>\delta\})\times\left({ W^{(1,2),2}}((t_1,t)\times\{M_0>\delta\})\cap H^1([0,t],\HOneO,\HmOne)\right),\nonumber\\
&K_{u_0}^{(\delta)}:{\cal B}\rightarrow Z_{u_0}^{(\delta)},\, K_{u_0}^{(\delta)}(u_{10})(s):=(S(s)u_{10})|_{\left\{M_0>\delta\right\}} \text{ for all }s\in[0,t],\,u_{10}\in {\cal B}.\nonumber
\end{align}
Observe that since $M_0$ is a continuous function, the level sets $\left\{M_0>\delta\right\}$ are open. Hence, the spaces $Y_{u_0}^{(\delta)}$ and  $Z_{u_0}^{(\delta)}$ are well defined.  

Now we can formulate a smoothing property for our case:
\begin{Theorem}[Smoothing property]\label{theoSP} Let $\Omega$ be a smooth bounded domain in $\R^N$, $N\in\{1,2,3\}$. 
Let the functions $f$ and $g$ satisfy  assumptions \eqref{fM}-\eqref{g2} and \eqref{tildefnew} and let the given constants $\alpha$ and $\gamma$ satisfy
$\frac{\alpha}{2}+1<\gamma<\alpha$. Then there exist some constants $\Cr{A2},\Cr{A3},\delta,\varepsilon,T>0$ depending only on the parameters of the problem and such that the following smoothing property holds for the operator $S(T)$:
\begin{subequations}\label{SP}
\begin{align}
&\|(S(T))(u_{10})-(S(T))(u_{20})\|_{X}\leq \frac{1}{2}\|u_{10}-u_{20}\|_{X}+\Cl[A]{A2}\left\|K_{u_0}^{\left(\frac{\delta}{2}\right)}(u_{10})-K_{u_0}^{\left(\frac{\delta}{2}\right)}(u_{20})\right\|_{Y_{u_0}^{(\delta)}},\label{SP1}\\
&\left\|K_{u_0}^{\left(\frac{\delta}{2}\right)}(u_{10})-K_{u_0}^{\left(\frac{\delta}{2}\right)}(u_{20})\right\|_{Z_{u_0}^{\left(\frac{\delta}{2}\right)}}\leq \Cl[A]{A3}\|u_{10}-u_{20}\|_{X}\label{SP2}\\
&\text{ for }\|u_0-u_{10}\|_{X},\,\|u_0-u_{20}\|_{X}\leq\varepsilon,\ u_0,u_{10},u_{20}\in{\cal B},
\end{align}
\end{subequations}
where ${\cal B}$ is the absorbing set from {\it Lemma \ref{absorb}}.
\end{Theorem}
\noindent { Due to the presence of the taxis term it seems impossible to handle the difference $M_1-M_2$ in $L^1$, as {\cb is} done for the porous medium equation in \cite[Chapter 4]{bookMe}. { This is the reason why} we use the $H^{-1}$-norm instead. It offers another convenient choice for a degenerate equation. We first recall some useful and well known facts about the gradient operator ($\nabla$), its adjoint ($\nabla^*$), and pseudo-inverse ($\nabla^+$):   
\begin{align}
&H^{1}_{0}(\Omega)\xtofrom[\nabla^{+}]{\nabla}(\LTwo)^N\xtofrom[\nabla^{+*}]{\nabla^{*}}\HmOne,\nonumber\\ &\nabla^{*}=-\nabla\cdot,\quad \nabla^+=\nabla^{-1}\Pi_{ \nabla\left(H^{1}_{0}(\Omega)\right)},\nonumber\\
&\nabla^+\nabla=id,\label{nab1}\\
&(-\Delta)^{-1}=\nabla^{+}\nabla^{+*}\label{Deltam}\\
&\nabla(-\Delta)^{-1}=\nabla^{+*},\label{nab12}\\
&\|\nabla^{+*}u^*\|_{(\LTwo)^N}=\|u^*\|_{\HmOne}.\label{nab2}
\end{align}
Here $\Pi_{\nabla\left(H^{1}_{0}(\Omega)\right)}$ denotes the  orthogonal projection on $\nabla\left(H^{1}_{0}(\Omega)\right)$ which is a closed subspace of  $(\LTwo)^N$.} 
{\it Proof of Theorem \ref{theoSP}.} Let us consider arbitrary points $(M_0,\rho_0),(M_{10},\rho_{10}),(M_{20},\rho_{20})\in {\cal B}$ such that
\begin{align}
	&\|(M_0-M_{10},\rho_0-\rho_{10})\|_{H^{-1}(\Omega)\times\LTwo},\ \|(M_0-M_{20},\rho_0-\rho_{20})\|_{H^{-1}(\Omega)\times\LTwo}\leq\varepsilon.\label{diste}
\end{align}
for some $\varepsilon>0$ which we will fix later on. Due to the interpolation inequality \eqref{SobInter} {\cb we obtain} that 
\begin{align}
 \|M_0-M_{10}\|_{\LInf},\ \|M_0-M_{20}\|_{\LInf}\leq \Cl{C21_}{\cb \varepsilon^{\theta_1}}.\label{diste1}
\end{align}
Let $(M,\rho),(M_{1},\rho_{1}),(M_{2},\rho_{2})$ be the corresponding solutions to system \eqref{PMCT}. 
Subtracting equation \eqref{EqI} for {\cb solutions $(M_{1},\rho_{1})$ and $(M_{2},\rho_{2})$} we obtain with \eqref{tildefnew} that
\begin{align}
	\partial_t(M_{1}-M_{2})=&\frac{1}{\alpha+1}\Delta \left(M_{1}^{\alpha+1}-M_{2}^{\alpha+1}\right)
	- \nabla\cdot\left((M_{1}^\gamma-M_{2}^\gamma) \nabla\rho_{2}\right)- \nabla\cdot\left(M_{1}^\gamma \nabla(\rho_{1}-\rho_{2})\right)\nonumber\\
	&-\Cr{F_5}(M_1-M_2)-\left(\widetilde{f}\left(M_1^{\beta},\rho_1\right)-\widetilde{f}\left(M_2^{\beta},\rho_2\right)\right).\label{differ}
\end{align}
Further, in order to shorten the notation we introduce the quantities
\begin{align}
&W:=M_{1}-M_{2},\quad {\cb W_0:=W(0)},\quad v:=\rho_{1}-\rho_{2},\quad {\cb v_0:=v(0)},\nonumber\\
&U_{\delta,\varepsilon,T}:={\max}\left\{M_{1}(s,x),M_{2}(s,x)|\, t\in[0,T],\, x\in \left\{M_0\leq\delta\right\}\right\}.\nonumber
\end{align}
 Our first goal is to progress {\cb towards} the 'contractive' part of estimate \eqref{SP1}. In order to achieve this we need to obtain some kind of dissipativity estimate for $\|M\|_{\HmOne}$  with  perturbation terms which do not contain  norms of $M$-component on level sets $\{M_0\leq\delta\}$. 
Multiplying \eqref{differ} by $(-\Delta)^{-1}W$ and  integrating  over $\Omega$, we arrive at
\begin{align}	
\frac{1}{2}\frac{d}{dt}\left\|\nabla^{+*}W\right\|_{\LTwon}^2=&-\frac{1}{\alpha+1}\left(M_{1}^{\alpha+1}-M_{2}^{\alpha+1}, M_1-M_2\right)_{\LTwo}+\left((M_{1}^\gamma-M_{2}^\gamma) \nabla\rho_{2},\nabla^{+*}W\right)_{\LTwon}\nonumber\\
&+\left(M_{1}^\gamma \nabla v,\nabla^{+*}W\right)_{\LTwon}-\Cr{F_5}\left\|\nabla^{+*}W\right\|_{\LTwon}^2\nonumber\\
&-\left(\widetilde{f}\left(M_1^{\beta},\rho_1\right)_{\LTwo}-\widetilde{f}\left(M_2^{\beta},\rho_2\right),\nabla^+ \nabla^{+*}W\right)_{\LTwo}.\label{differ1}
\end{align}
Here we used the definition of the adjoint and properties   \eqref{Deltam}-\eqref{nab12}. 
Using the inequalities
\begin{align}	&\left(M_{1}^{\alpha+1}-M_{2}^{\alpha+1}\right)(M_{1}-M_{2})\geq {\cb  \frac{\alpha+1}{\left(1+\frac{\alpha}{2}\right)^2}} \left(M_{1}^{1+\frac{\alpha}{2}}-M_{2}^{1+\frac{\alpha}{2}}\right)^2,\label{low1}\\	&\frac{c}{b}\max\left\{M_{1},M_{2}\right\}^{\frac{c-b}{b}}\left|M_{1}^{b}-M_{2}^{b}\right|\geq  
	\left|M_{1}^{c}-M_{2}^{c}\right| \text{ for all } c\geq b>0,\label{low2}
\end{align}
and assumptions $\gamma,\beta>1+\frac{\alpha}{2}$, we can estimate the terms on the right-hand side of \eqref{differ1} in the following way:
\begin{align}
 &-\left(M_{1}^{\alpha+1}-M_{2}^{\alpha+1},M_{1}-M_{2}\right)_{\LTwo}\leq -{\cb \C}\left\|M_{1}^{1+\frac{\alpha}{2}}-M_{2}^{1+\frac{\alpha}{2}}\right\|_{\LTwo}^2,\label{est6}
\end{align}
\begin{align}
 &\left|\left((M_{1}^\gamma-M_{2}^\gamma) \nabla\rho_{2},\nabla^{+*}W\right)_{\LTwon}\right|\nonumber\\
 \leq&\C\left\|\nabla^{+*}W\right\|_{\LTwon}\left(\left\|M_{1}^\gamma-M_{2}^\gamma\right\|_{L^2(\{M_0\leq \delta\})}+\left\|M_{1}^\gamma-M_{2}^\gamma\right\|_{L^2(\{M_0> \delta\})}\right)\nonumber\\
 \leq&\C \left\|\nabla^{+*}W\right\|_{\LTwon}\left(U_{\delta,\varepsilon,T}^{\frac{2\gamma-2-\alpha}{\alpha+2}}\left\|M_{1}^{1+\frac{\alpha}{2}}-M_{2}^{1+\frac{\alpha}{2}}\right\|_{\LTwo}+\|W\|_{L^2(\{M_0> \delta\})}\right),
 \end{align}
 \begin{align}
 \left|\left(M_{1}^\gamma \nabla v,\nabla^{+*}W\right)_{\LTwon}\right|\leq&\C\left\|\nabla^{+*}W\right\|_{\LTwon}\left(\left\|M_{1}^\gamma \nabla v\right\|_{L^2(\{M_0\leq \delta\})}+\left\|M_{1}^\gamma \nabla v\right\|_{(L^2(\{M_0> \delta\}))^n}\right)\nonumber\\
 \leq&\C \left\|\nabla^{+*}W\right\|_{\LTwon}\left(U_{\delta,\varepsilon,T}^{\gamma}\left\|\nabla v\right\|_{\LTwon}+\left\|\nabla v\right\|_{(L^2(\{M_0> \delta\}))^n}\right),
\end{align}
\begin{align}
&\left|\left(\widetilde{f}\left(M_{1}^{\beta},\rho_1\right)-\widetilde{f}\left(M_{2}^{\beta},\rho_2\right),\nabla^+ \nabla^{+*}W\right)_{\LTwo}\right|\nonumber\\
\leq &\Cl{C11}\left\|\nabla^{+*}W\right\|_{\LTwon}\left(\left\|M_{1}^{\beta}-M_{2}^{\beta}\right\|_{\LTwo}+\|v\|_{\LTwo}\right)\nonumber\\
\leq &\Cr{C11}\left\|\nabla^{+*}W\right\|_{\LTwon}\left(\left\|M_{1}^\beta-M_{2}^\beta\right\|_{L^2(\{M_0\leq \delta\})}+\left\|M_{1}^\beta-M_{2}^\beta\right\|_{L^2(\{M_0> \delta\})}+\|v\|_{\LTwo}\right)\nonumber\\
\leq&\C \left\|\nabla^{+*}W\right\|_{\LTwon}\left(U_{\delta,\varepsilon,T}^{\frac{2\beta-2-\alpha}{\alpha+2}}\left\|M_{1}^{1+\frac{\alpha}{2}}-M_{2}^{1+\frac{\alpha}{2}}\right\|_{\LTwo}+\left\|W\right\|_{L^2(\{M_0> \delta\})}+\|v\|_{\LTwo}\right).\label{est4}
\end{align}
Combining \eqref{differ1} and  \eqref{est6}-\eqref{est4} with the Young inequality, we conclude that
\begin{align}
	&\frac{1}{2}\frac{d}{dt}\left\|\nabla^{+*}W\right\|_{\LTwon}^2\nonumber\\
	\leq &\left(-\frac{\Cr{F_5}}{2}+\C\left(U_{\delta,\varepsilon,T}^{2\frac{2\gamma-2-\alpha}{\alpha+2}}+U_{\delta,\varepsilon,T}^{2\frac{2\beta-2-\alpha}{\alpha+2}}\right)\right)\left\|\nabla^{+*}W\right\|_{\LTwon}^2+\Cl{B2}U_{\delta,\varepsilon,T}^{2\gamma}\left\|\nabla v\right\|_{\LTwon}^2\nonumber\\	&+\Cl{C12}\left(\left\|W\right\|_{L^2(\{M_0> \delta\})}^2+\left\|\nabla v\right\|_{(L^2(\{M_0> \delta\}))^n}^2+\|v\|_{\LTwo}^2\right)\nonumber\\
\leq &\left(-\frac{\Cr{F_5}}{2}+\Cl{B1}U_{\delta,\varepsilon,T}^{\kappa}\right)\left\|\nabla^{+*}W\right\|_{\LTwon}^2+\Cr{B2}U_{\delta,\varepsilon,T}^{\kappa}\left\|\nabla v\right\|_{\LTwon}^2\nonumber\\
&+\Cl{C16}\left(\left\|W\right\|_{L^2(\{M_0> \delta\})}^2+\left\|\nabla v\right\|_{(L^2(\{M_0> \delta\}))^n}^2+\|v\|_{\LTwon}^2\right),\label{est1}
\end{align}
where 
\begin{align}	\kappa:=2\min\left\{\frac{2\min\left\{\gamma,\beta\right\}-2-\alpha}{\alpha+2},\gamma\right\}=2\frac{2\min\left\{\gamma,\beta\right\}-2-\alpha}{\alpha+2}>0\label{kappa}
\end{align}
due to the assumptions on $\alpha$,  $\beta$, and $\gamma$.  { We emphasise at this point that our new and sharper assumptions $\gamma,\beta>1+\frac{\alpha}{2}$ allow not only to absorb the $L^2$ norms of the differences of some powers of $M_1$ and $M_2$  coming from the taxis and reaction terms, but also to obtain some dissipativity with respect to $\left\|\nabla^{+*}W\right\|_{\LTwon}^2$ in \eqref{est1}. Indeed, as we will see later, $U_{\delta,\varepsilon,T}$ can be made arbitrary small by choosing $\varepsilon$ and $\delta$ 
sufficiently small. Thanks to $\kappa>0$ this leads to a negative coefficient $\left(-\frac{\Cr{F_5}}{2}+\Cr{B1}U_{\delta,\varepsilon,T}^{\kappa}\right)$. With $\gamma$ or $\beta$ equal to $1+\frac{\alpha}{2}$ we would have $\kappa=0$  instead and thus a potentially positive coefficient $\left(-\frac{\Cr{F_5}}{2}+\Cr{B1}\right)$ which only guaranties an estimate such as \eqref{Lip}. 

Leaving  \eqref{est1}} for while we will now establish an bound for $\left\| v\right\|$. {  Once again,  we need to take care so as not to include any $L^p(\{M_0\leq\delta\})$-norms of  $W$  in our estimates.} Subtracting  equation \eqref{EqII} for $(M_{1},\rho_{1})$ and $(M_{2},\rho_{2})$  we obtain that
\begin{align}
	\partial_tv=\Delta v-\Cr{G_1}v-(g_2(\rho_{1})-g_2(\rho_{2}))M_{1}-g_2(\rho_{2})W.\label{diffro}
\end{align}
Multiplying \eqref{diffro} by $v$, integrating  over $\Omega$, and using \eqref{nab1}, we obtain that
\begin{align}
	\frac{1}{2}\frac{d}{dt}\left\| v\right\|_{\LTwo}^2
	=& -\left\|\nabla v\right\|_{\LTwon}^2-\Cr{G_1}\left\| v\right\|_{\LTwo}^2-\left(g_2(\rho_{1})-g_2(\rho_{2}),M_{1} v\right)_{\LTwo}-\left(W,g_2(\rho_2)v\right)_{\LTwo}\nonumber\\
	=& -\left\|\nabla v\right\|_{\LTwon}^2-\Cr{G_1}\left\| v\right\|_{\LTwo}^2-\left(g_2(\rho_{1})-g_2(\rho_{2}),M_{1} v\right)_{\LTwo}\nonumber\\
&-\left(\nabla^{+*}W,g_2(\rho_{2})\nabla  v + v \frac{d g_2}{d \rho}(\rho_{2})\nabla  \rho_{2}\right)_{\LTwon}.\label{diffro1}
\end{align}
Using the assumptions on  $g_2$ and the Young inequality 
we conclude from \eqref{diffro1} that
\begin{align}
\frac{1}{2}\frac{d}{dt}\left\| v\right\|_{\LTwo}^2
\leq& -\left\|\nabla v\right\|_{\LTwon}^2-\Cr{G_1}\left\| v\right\|_{\LTwo}^2+\C\left(M_{1},v^2\right)_{\LTwo}\nonumber\\
&+\left\|\nabla^{+*}W\right\|_{\LTwon}\left(\|\nabla v\|_{\LTwon} + \|v \|_{\LTwo}\right)\nonumber\\
\leq &\Cl{B6}\left\|\nabla^{+*}W\right\|_{\LTwon}^2-\frac{1}{2}\left\|\nabla  v\right\|_{\LTwon}^2+\C\|v\|_{\LTwo}^2.\label{est2}
\end{align}
Let us now multiply the inequality (\ref{est2}) by the constant $\Cl{B9}:=\frac{\Cr{F_5}}{4\Cr{B6}}$ and add it to (\ref{est1}). This yields
\begin{align}
&\frac{1}{2}\frac{d}{dt}\left(\left\|\nabla^{+*}W\right\|_{\LTwon}^2+\Cr{B9}\left\| v\right\|_{\LTwo}^2\right)\nonumber\\
\leq  &\left(-\frac{\Cr{F_5}}{4}+\Cr{B1}U_{\delta,\varepsilon,T}^{\kappa}\right)\left\|\nabla^{+*}W\right\|_{\LTwon}^2
+\left(-\Cl{C18}+\Cr{B2}U_{\delta,\varepsilon,T}^{\kappa}\right)\left\|\nabla v\right\|_{\LTwon}^2+\Cl{C17}\|v\|_{\LTwo}^2\nonumber\\
&+\Cr{C16}\left(\left\|W\right\|_{L^2(\{M_0> \delta\})}^2+\left\|\nabla v\right\|_{(L^2(\{M_0> \delta\}))^n}^2\right).\label{est1_2}
\end{align}
Observe that due to \eqref{Holder} and \eqref{diste} it follows for all  $\delta\in(0,\|M_0\|_{L^{\infty}(\Omega)})$ that 
\begin{align}
  U_{\delta,\varepsilon,T}\leq\delta+ L(T,R)\varepsilon^{\theta_{\infty}}.
  \label{Ue}
\end{align}
Combining \eqref{est1_2}-\eqref{Ue} and recalling that $\kappa>0$ due to \eqref{kappa}, we conclude that 
if $\delta$ and ${\cb \varepsilon}$ are chosen in such a way that
\begin{align}
 \delta+ L(T,R)\varepsilon^{\theta_{\infty}}\leq \min\left\{\frac{\Cr{F_5}}{8\Cr{B1}},\frac{\Cr{C18}}{2\Cr{B2}}\right\}^{\frac{1}{\kappa}},\label{Uedt1}
\end{align}
then
\begin{align}
\frac{1}{2}\frac{d}{dt}\left(\left\|\nabla^{+*}W\right\|_{\LTwon}^2+\Cr{B9}\left\| v\right\|_{\LTwo}^2\right)
\leq  &-\frac{\Cr{F_5}}{8}\left\|\nabla^{+*}W\right\|_{\LTwon}^2
-\frac{\Cr{C18}}{2}\left\|\nabla v\right\|_{\LTwon}^2+\Cr{C17}\|v\|_{\LTwo}^2\nonumber\\
&+\Cr{C16}\left(\left\|W\right\|_{L^2(\{M_0> \delta\})}^2+\left\|\nabla v\right\|_{(L^2(\{M_0> \delta\}))^n}^2\right)\nonumber\\
\leq  &-\Cl{C20} \left(\left\|\nabla^{+*}W\right\|_{\LTwon}^2+\Cr{B9}\left\| v\right\|_{\LTwo}^2\right)\nonumber\\
&+\Cl{C19}\left(\left\|W\right\|_{L^2(\{M_0> \delta\})}^2+\left\|\nabla v\right\|_{(L^2(\{M_0> \delta\}))^n}^2+\|v\|_{\LTwo}^2\right).\label{est1_3_}
\end{align}
Using the Gronwall lemma, the Lipschitz property \eqref{Lip}, and property \eqref{nab2} we obtain with \eqref{est1_3_} that for all $0<t_1<T$
\begin{align}
 &\|W(T)\|_{\HmOne}^2+\Cr{B9}\left\| v(T)\right\|_{\LTwo}^2 \nonumber\\
 \leq& e^{-2\Cr{C20}(T-t_1)}\left(\|W(t_1)\|_{\HmOne}^2+\Cr{B9}\left\|v(t_1)\right\|_{\LTwo}^2\right)\nonumber\\
 &+2\Cr{C19}\int_{t_1}^Te^{2\Cr{C20}(s-T)}\left(\left\|W\right\|_{L^2(\{M_0> \delta\})}^2+\left\|\nabla v\right\|_{(L^2(\{M_0> \delta\}))^n}^2+\|v\|_{\LTwo}^2\right)\,d s\nonumber\\
 \leq& \C L_0^2(t_1,R)e^{-2\Cr{C20}(T-t_1)}\left(\|W_0\|_{\HmOne}^2+\Cr{B9}\left\|v_0\right\|_{\LTwo}^2\right)\nonumber\\
 &+2\Cr{C19}\int_{t_1}^T\left\|W\right\|_{L^2(\{M_0> \delta\})}^2+\left\|\nabla v\right\|_{(L^2(\{M_0> \delta\}))^n}^2+\|v\|_{\LTwo}^2\,d s,\label{int1_2_}
\end{align}
which finally leads to the estimate
\begin{align}
 \|(W(T),v(T))\|_{X}\leq& \Cl{C27} L_0(t_1,R)e^{-\Cr{C20}(T-t_1)}\|(W_0,v_0)\|_{X}+\Cl{C24}\left\|(W,v)\right\|_{Y_{u_0}^{(\delta)}}\nonumber\\
 \leq&\frac{1}{2}\|(W_0,v_0)\|_{X}+\Cr{C24}\left\|(W,v)\right\|_{Y_{u_0}^{(\delta)}}\label{int1_3_}
 \end{align}
if $t_1$ and $T$ are such that
\begin{align}
 \Cr{C27} L_0(t_1,R)e^{-\Cr{C20}(T-t_1)}\leq\frac{1}{2}.\label{contr}
\end{align}

Next, we study the pair $(W,v)$ on the sets $\left\{M_0>\delta\right\}$. { On these sets the equation for $M$ is non-degenerate, which allows to use standard estimates for uniformly parabolic PDEs and thus obtain  better regularity.}
Starting once again with equations \eqref{differ} and \eqref{diffro}, we now rewrite them in the following way:
{\cb
\begin{subequations}\label{systemWv}
\begin{align}
 &\partial_t W={\cb M_1^{\alpha}}\Delta W+\Cl[b]{bb1}\cdot\nabla W+\Cl[b]{bb2} W+\Cl[b]{bb4}\Delta v+\Cl[b]{bb3}\cdot\nabla v+\Cl[b]{bb45}v,\label{diffMnond}\\
 &\partial_t v=\Delta v+\Cl[b]{b5}v+\Cl[b]{b6}W,\label{v}
\end{align}
\end{subequations}
where 
\begin{align*}
 \Cr{bb1}:=&\alpha M_1^{\alpha-1}\nabla (M_1+M_2)-\gamma M_2^{\gamma-1}\nabla\rho_1,\\
 \Cr{bb2}:=&\frac{M_1^{\alpha}-M_2^{\alpha}}{M_1-M_2}\Delta M_2+\alpha \frac{M_1^{\alpha-1}-M_2^{\alpha-1}}{M_1-M_2}|\nabla M_2|^2-\frac{M_1^{\gamma}-M_2^{\gamma}}{M_1-M_2}\Delta\rho_1-\gamma\frac{M_1^{\gamma-1}-M_2^{\gamma-1}}{M_1-M_2}\nabla M_1\cdot\nabla\rho_1,\\
 &-\frac{f(M_1,\rho_1)-f(M_2,\rho_1)}{M_1-M_2},\\
 \Cr{bb4}:=&-M_2^{\gamma},\\
 \Cr{bb3}:=&-\gamma M_2^{\gamma-1}\nabla M_2,\\
 \Cr{bb45}:=&-\frac{f(M_2,\rho_1)-f(M_2,\rho_2)}{\rho_1-\rho_2},\nonumber\\
 \Cr{b5}:=&-\frac{g(M_1,\rho_1)-g(M_1,\rho_2)}{\rho_1-\rho_2},\\
 \Cr{b6}:=&-\frac{g(M_1,\rho_2)-g(M_2,\rho_2)}{M_1-M_2}.
\end{align*}
 } Observe that due to \eqref{diste1}   we have for all $\delta_0\in\left(0,\|M_0\|_{\LInf}\right)$ that
\begin{align}
 {\inf}\left\{M_{10}(x),M_{20}(x)|\, x\in \left\{M_0>\delta_0\right\}\right\}
 \geq &\delta_0-\Cr{C21_}\varepsilon^{{\cb \theta_1}}\nonumber\\
 \geq &\frac{\delta_0}{2},\label{posdel}
\end{align}
so that, due to \eqref{MPos1_},
\begin{align}
 \inf\left\{M(t,x)|\ s\in[0,t],\,x\in\left\{M_0>\delta_0\right\}\right\}\geq \Cr{C26_}\left(\frac{\delta_0}{2},t\right),\label{MPos1_2}
\end{align}
if
\begin{align}
 \Cr{C21_}\varepsilon^{{\cb \theta_1}}\leq\frac{\delta_0}{2}.\label{epsdelta}
\end{align}
{\cb Thus, for such $\delta_0$ and  $\varepsilon$ 
 system \eqref{diffMnond} is a nondegenerate linear parabolic  system w.r.t. $(W,v)$. Moreover,  coefficients $b_i$  are compositions of continuous functions with  $M_k$ and $\rho_k$, $k=1,2$, and their partial derivatives up to the second order. {\it Lemma \ref{absorb}} implies that $b_i$'s all belong to $L^{\infty}((0,T)\times \{M_0>\delta_0\})$. Altogether, standard results on interior regularity}   in Sobolev spaces (see, e.g., Theorems 9.1 and 10.1, and the remark on local estimates in Sobolev spaces at the end of \textsection 10 in \cite[Chapter IV]{LSU}) together with estimates from {\it Lemma \ref{absorb}} imply that for all $0<t_0<t_1<T$ and  $\delta_0<\delta_1<\delta_2$ it holds 
\begin{align}
&\|W\|_{{ W^{(1,2),2}}((t_1,T)\times\{M_0>\delta_2\})}\leq \C(\delta_1,\delta_2,t_1,T)\left(\|W\|_{L^{2}((t_0,T)\times\{M_0>\delta_1\})}+\|v\|_{{ W^{(1,2),2}}((t_0,T)\times\{M_0>\delta_1\})}\right),\label{estW1}\\
 &\|v\|_{{ W^{(1,2),2}}((t_0,T)\times\{M_0>\delta_1\})}\leq \C(\delta,\delta_1,t_0,T)\left(\|v\|_{L^{2}((0,T)\times\{M_0>\delta_0\})}+\|W\|_{L^{2}((0,T)\times\{M_0>\delta_0\})}\right).\label{estv1}
\end{align}
Plugging  \eqref{estv1} into \eqref{estW1}, we obtain that
\begin{align}
 \|W\|_{{ W^{(1,2),2}}((t_1,T)\times\{M_0>\delta_2\})}
 \leq &\C(\delta,\delta_1,\delta_2,t_0,t_1,t)\|(W,v)\|_{L^{2}([0,t]\times\{M_0>\delta_0\})}.\label{estW2}
\end{align}
Thus, choosing $\delta_0:=\frac{\delta}{4}$, $\delta_1:=\frac{3\delta}{8}$, $\delta_2:=\frac{\delta}{2}$ in \eqref{estW2} and  $\delta_0:=\frac{\delta}{4}$, $\delta_1:=\frac{\delta}{2}$, $t_0:=t_1$ in \eqref{estv1} yields
\begin{align}
 \|(W,v)\|_{({ W^{(1,2),2}}((t_1,T)\times\{M_0>\frac{\delta}{2}\}))^2}
 \leq &\Cl{C104}(\delta,t_1,T)\|(W,v)\|_{(L^{2}((0,T)\times\{M_0>\frac{\delta}{4}\}))^2}.\label{estWv1}
\end{align}
Combining \eqref{estWv1} with  \eqref{Lip}, \eqref{MPos1_2}, and the inequalities \eqref{low1} and 
\begin{align}
 |M_1-M_2|\leq \frac{2}{\alpha+2}\inf\{M_1,M_2\}^{-\frac{\alpha}{2}}\left|M_{1}^{1+\frac{\alpha}{2}}-M_{2}^{1+\frac{\alpha}{2}}\right|,
\end{align}
we thus arrive at the estimate
\begin{align}
 &\|(W,v)\|_{({ W^{(1,2),2}}((t_1,T)\times\{M_0>\frac{\delta}{2}\}))^2}\nonumber\\
 \leq &\C(\delta,t_1,T)\left(\int_0^{\cb T}\left(M_{1}^{\alpha+1}-M_{2}^{\alpha+1}, M_1-M_2\right)\,ds\right)^{\frac{1}{2}}+{\cb\Cr{C104}(\delta,t_1,T)\|v\|_{L^2((0,T)\times\Omega)}}\nonumber\\
 \leq &\C(\delta,t_1,T)\|(W_0,v_0)\|_{H^{-1}(\Omega)\times\LTwo}.\label{estWv2}
\end{align}
Next, we recall that due to \eqref{Lip} it holds that
\begin{align}
 \|v\|_{L^2(0,T;\HOneO)}\leq L_0(T,R)\|(W_0,v_0)\|_{H^{-1}(\Omega)\times\LTwo}.\label{estnab}
\end{align}
Going back to \eqref{diffro}, we compute  that
\begin{align}
 &\|\partial_t v\|_{H^{-1}(\Omega)}\nonumber\\
 =&\|\Delta v-\Cr{G_1}v-(g_2(\rho_{1})-g_2(\rho_{2}))M_{1}-g_2(\rho_{2})W\|_{H^{-1}(\Omega)}\nonumber\\
 \leq&{\cb\|v\|_{\HOneO}+\Cr{G_1}\|v\|_{H^{-1}(\Omega)}+\C\|M_{1}\|_{\LInf}\|g_2(\rho_{1})-g_2(\rho_{2})\|_{\LTwo}+\C\|g_2(\rho_{2})\|_{W^{1,\infty}(\Omega)}\|W\|_{H^{-1}(\Omega)}}\nonumber\\
 \leq&\C{\cb(R)}\left(\|v\|_{\HOneO}+\|W\|_{H^{-1}(\Omega)}\right).\label{dtv}
\end{align}
Integrating \eqref{dtv} over $(0,t)$ and combining with \eqref{Lip} and \eqref{estnab}, we finally obtain that
\begin{align}
 \|v\|_{H^1((0,T),\HOneO,\HmOne)}\leq\C({\cb R},T)\|(W_0,v_0)\|_{H^{-1}(\Omega)\times\LTwo}.\label{vpar}
\end{align}
With \eqref{int1_3_}, \eqref{estWv2}, \eqref{vpar} we  have the conditions of the smoothing property \eqref{SP}, it only remains to choose the parameters in such a way that conditions \eqref{Uedt1}, \eqref{contr}, and \eqref{epsdelta} (recall that $\delta_0=\frac{\delta}{4}$) are satisfied, i.e., if
\begin{align}
 &\delta+ L(T,R)\varepsilon^{{\cb \theta_1}}\leq \min\left\{\frac{\Cr{F_5}}{8\Cr{B1}},\frac{\Cr{C18}}{2\Cr{B2}}\right\}^{\frac{1}{\kappa}},\label{cond1}\\
 &\Cr{C23} L_0(t_1,R)e^{-\Cr{C20}(T-t_1)}\leq\frac{1}{2},\label{cond2}\\
 &\Cr{C21_}\varepsilon^{{\cb \theta_1}}\leq\frac{\delta}{8}.\label{cond3}
\end{align}
Clearly, the exist such $t_1,T,\delta,$ and $\varepsilon$, that  conditions \eqref{cond1}-\eqref{cond3} are satisfied. Indeed, for any $t_1>0$ one can choose $T$ large enough so as to fulfil \eqref{cond2}. Then, choosing $\delta:=8\Cr{C21_}\varepsilon^{{\cb \theta_1}}$ in order to comply with \eqref{cond3}, it remains to choose $\varepsilon$ so small as to meet \eqref{cond1}.   
{\it Theorem~\ref{theoSP}} is  proved. 
\\\\\mbox{}\hfill\ensuremath{\Box}
\section{Proof of Theorem~\ref{mainexpo}}\label{mainproof}
We are finally ready for the\\\\
{\it Proof of Theorem~\ref{mainexpo}}. { Our proof goes through the following steps. First, we prove the existence of an exponential attractor ${\cal M}$ in  $H^{-1}(\Omega)\times\LTwo$-metric. This we achieve with the help of the smoothing property \eqref{SP}. Finally, we use the Sobolev interpolation inequality in order to show that ${\cal M}$ is {\cb  at the same time} an exponential attractor in  $\X$-metric{\cb .} 

}
Due to {\it Theorem~\ref{theoSP}} the exists a number $T>0$ such that for $S(T)$ the smoothing property \eqref{SP} holds. The existence of an exponential attractor for the {\it discrete} semigroup $S(nT), n\in\N${\cb ,} in the set ${\cal B}\subset X$ is a consequence of Remark 4.3 of \cite{bookMe}, we only need to verify that  $Z_{u_0}^{\left(\frac{\delta}{2}\right)}$ is uniformly (w.r.t. $u_0\in{\cal B}$) compactly embedded in  $Y_{u_0}^{(\delta)}$. 
Due to Lions-Aubin lemma, we have that
\begin{align}
H^1\left((t_1,T),H^{1}_{0}(\Omega),H^{-1}(\Omega)\right)
\subset\subset & L^2\left((t_1,T),\LTwo\right).\label{Emb2}
\end{align}
Therefore, we only need to study the (obviously continuous) canonical embedding
\begin{align*}
 i_{u_0}:{ W^{(1,2),2}}\left((t_1,t)\times\left\{M_0>{\delta}/{2}\right\}\right)\rightarrow L^2((t_1,t),H^1(\{M_0>\delta\})),\quad i_{u_0}u:=u|_{(t_1,t)\times \{M_0>\delta\}}.
\end{align*}
We proceed similar to \cite[Proposition A.5]{EfZelik}, where the case of a H\"older space embedded in the space of continuous functions on a smaller domain was considered. Let us define for each $u_0\in{\cal B}$ an extension operator
\begin{align*}
 &p_{u_0}: { W^{(1,2),2}}\left((t_1,t)\times\left\{M_0>{\delta}/{2}\right\}\right)\rightarrow { W^{(1,2),2}}\left((t_1,t)\times\Omega\right),\quad p_{u_0}u:=\begin{cases}\varphi_{M_0}u&\text{in }  \left\{M_0>{\delta}/{2}\right\},\\
 0&\text{in }\left\{M_0\leq\frac{\delta}{2}\right\}.                                                                                                                                                       \end{cases}
\end{align*}
Here $\varphi_{M_0}$ is any cutoff function which satisfies  \eqref{cutoff} for $\delta_0:=\frac{3\delta}{4}$ and $\delta_1:=\delta$. Since  $\varphi_{M_0}$ is a test function and compactly supported in $\left\{M_0>{\delta}/{2}\right\}$,  it follows that $\{p_{u_0}\}_{u_0\in{\cal B}}$ is a family of well defined continuous linear operators. Moreover, even though these operators are defined on different spaces, their norms are uniformly bounded:
\begin{align}
 \|p_{u_0}\|\leq A_p\quad\text{for all }u_0\in{\cal B}\label{Ap}
\end{align}
for some constant $A_p>0$. This is a consequence of property \eqref{nablaphi}. 
Note also that our choice of  cutoff function guaranties that
\begin{align}
 p_{u_0}u=u\quad\text{in }(t_1,t)\times\{M_0>\delta\}.\label{tranc}
\end{align}
Next, we define a restriction operator
\begin{align*}
 c_{u_0}:L^2((t_1,t),H^1(\Omega))\rightarrow L^2((t_1,t),H^1(\{M_0>\delta\})),\quad c_{u_0}u=u|_{(t_1,t)\times\{M_0>\delta\}}. 
\end{align*}
In this case, the value ranges depend upon $u_0$, but, clearly, 
\begin{align}
\|c_{u_0}\|\leq1\quad\text{for all }u_0\in{\cal B}.\label{c1}
\end{align}
Finally, we recall that due to the Lions-Aubin lemma the canonical embedding 
\begin{align*}
j:{ W^{(1,2),2}}\left((t_1,t)\times\Omega\right)\rightarrow L^2((t_1,t),H^1(\Omega)), \quad ju=u                                                                                                                              \end{align*}
is compact. Observe that due to \eqref{tranc}
\begin{align}
 i_{u_0}=c_{u_0}jp_{u_0}.
\end{align}
Using \eqref{Ap} and \eqref{c1}, we compute that
\begin{align}
&N_{r}\left(i_{u_0}\left(B\left(0,1;{ W^{(1,2),2}}\left((t_1,t)\times\left\{M_0>{\delta}/{2}\right\}\right)\right)\right);L^2((t_1,t),H^1(\{M_0>\delta\}))\right)\nonumber\\
=& N_{r}\left(c_{u_0}jp_{u_0}\left(B\left(0,1;{ W^{(1,2),2}}\left((t_1,t)\times\left\{M_0>{\delta}/{2}\right\}\right)\right)\right);L^2((t_1,t),H^1(\{M_0>\delta\}))\right)\nonumber\\
\leq & N_{r}\left(jp_{u_0}\left(B\left(0,1;{ W^{(1,2),2}}\left((t_1,t)\times\left\{M_0>{\delta}/{2}\right\}\right)\right)\right);L^2((t_1,t),H^1(\Omega))\right)\nonumber\\
\leq & N_{r}\left(j\left(B\left(0,A_p;{ W^{(1,2),2}}\left((t_1,t)\times\Omega\right)\right)\right);L^2((t_1,t),H^1(\Omega))\right),\label{Nrbnd}
\end{align}
where $B(0,1;V)$ denotes the unit ball in a normed space $V$, and $N_{r}(C;V)$ denotes the minimum number of balls of radius $r>0$ needed in order to cover a compact set $C\subset V$. Since the bound  on the right-hand side of \eqref{Nrbnd} is independent of $u_0$, the embedding family $\{i_{u_0}\}$ is indeed uniformly compact. Due to the above observation this carries over to the embedding $Z_{u_0}^{\left(\frac{\delta}{2}\right)}\subset\subset Y_{u_0}^{(\delta)}$. 
Therefore, with Remark 4.3 from \cite{bookMe} we conclude that there exists an exponential attractor ${\cal M}_T$ for the semigroup $S(nT), n\in\N${\cb ,} in ${\cal B}$ (equipped with the $H^{-1}(\Omega)\times\LTwo$-topology) and its dimension and attraction parameters depend only upon the parameters of the problem. 
As usual (see, e.g., \cite[Remark 3.2]{bookMe}), the required exponential attractor ${\cal M}\subset {\cal B}$ for the continuous-time semigroup $S(t), t\geq0$ can be defined via 
\begin{align*}
 {\cal M}:=\bigcup_{t\in[0,T]}S(t){\cal M}_T\subset {\cal B}.
\end{align*}
For this construction to work, it suffices (compare \cite[Remark 3.2]{bookMe}) to check that the map $(t,u_0)\mapsto S(t)u_0$ is, say, H\"older continuous on $[0,T]\times {\cal B}$.
{\cb The H\"older continuity w.r.t. $t$} is a consequence of the regularity result  \eqref{HoelB} and the Sobolev embedding theorem. The Lipschitz continuity with respect to $u_0$ is given by  the Lipschitz property \eqref{Lip}. In both cases such parameters as the H\"older/Lipschitz constants and the H\"older exponent can be chosen to depend upon the parameters of the problem. Consequently, ${\cal M}$ is indeed an exponential attractor for $S(t)$ in ${\cal B}$ equipped with the $H^{-1}(\Omega)\times\LTwo$-topology and its dimension and attraction parameters depend only upon the parameters of the problem.

Finally, we observe that due to the interpolation inequalities \eqref{SobInter}-\eqref{SobInter2} the canonical embedding of ${\cal B}$ equipped with (the norm-induced) $H^{-1}(\Omega)\times\LTwo$-metric and ${\cal B}$ equipped with (the norm-induced) $\X$-metric is H\"older continuous. Consequently, ${\cal M}$ is an exponential attractor for $S(t)$ in ${\cal B}$ equipped with $\X$-metric and, once again, its dimension and attraction parameters depend only upon the parameters of the problem. Combining this with the fact that  ${\cal B}$ is an exponentially absorbing set in $\X$ and its diameter and absorption parameters depend only upon the parameters of the problem, we conclude that ${\cal M}$ is an exponential attractor for $S(t)$ in $\X$ and its dimension and attraction parameters depend only upon the parameters of the problem, as required.
{\it Theorem~\ref{mainexpo}} is thus proved.
\\\\\mbox{}\hfill\ensuremath{\Box}
\phantomsection
\printbibliography
\end{document}